\documentclass{amsart}[12]

\usepackage{geometry}
\usepackage[dvipsnames]{xcolor}
\usepackage{amsmath}
\usepackage{amsthm}
\usepackage{mathrsfs}
\usepackage{amsfonts}
\usepackage{ amssymb }
\usepackage{amsmath,amscd}
\usepackage{array}
\usepackage{mathbbol}
\usepackage{amssymb}
\usepackage[all, cmtip]{xy}
\usepackage{tikz}
\usepackage[shortlabels]{enumitem}
\usepackage{tensor}
\usepackage{ stmaryrd }
\usetikzlibrary{arrows}
\usetikzlibrary{cd}
\usepackage{hyperref}
\usepackage{MnSymbol}

\usepackage[backend=biber,
style=alphabetic,
isbn=false,
url=false,
sorting=nyt]{biblatex}
\usepackage{bm}
\renewcommand{\sslash}{\mathord{/\mkern-6mu/}}

\newtheorem{theorem}{Theorem}[subsection]
\newtheorem{lemma}[theorem]{Lemma}

\newtheorem{proposition}[theorem]{Proposition}
\newtheorem{definition}[theorem]{Definition}
\newtheorem{corollary}[theorem]{Corollary}
\newtheorem{remark}[theorem]{Remark}

\theoremstyle{definition}
\newtheorem{example}[theorem]{Example}

\newcommand{\CC}{\mathbb{C}}
\newcommand{\QQ}{\mathbb{Q}}

\newcommand{\NN}{\mathbb{N}}
\newcommand{\ZZ}{\mathbb{Z}}
\renewcommand{\AA}{\mathbb{A}}

\newcommand{\degen}{\mathrm{degen}}

\newcommand{\bc}{\mathbf{c}}

\newcommand{\sW}{\mathscr{W}}
\newcommand{\Gm}{\mathbb{G}_m}

\newcommand{\rk}{\mathrm{rank}}

\renewcommand{\sslash}{\mathord{/\mkern-6mu/}}

\newcommand{\bd}{\mathbf{d}}

\newcommand{\supp}{\mathrm{supp}}

\newcommand{\sst}{\textrm{ss}}
\newcommand{\sD}{\mathscr{D}}
\newcommand{\cW}{\mathscr{W}}

\DeclareMathOperator{\Hom}{Hom}

\DeclareMathOperator{\Cone}{Cone}

\newcommand{\Char}{\chi}

\newenvironment{Comment}[2]{\noindent\color{#1}{\texttt #2: }}{}

\begin{document}

\title{Weyl-invariant subspaces are (usually) not generic}
%\subtitle{or, abelianization of stability conditions}
\author{Riku Kurama, Ruoxi Li, Henry Talbott, and Rachel Webb}

\begin{abstract}
 Let $V$ be a linear representation of a connected complex reductive group $G$. Given a choice of character $\theta$ of $G$, Geometric Invariant Theory defines a locus $V^{ss}_\theta(G) \subseteq V$ of semistable points. We give necessary, sufficient, and in some cases equivalent conditions for the existence of $\theta$ such that a maximal torus $T$ of $G$ acts on $V^{ss}_\theta(T)$ with finite stabilizers. In such cases, the stack quotient $[V^{ss}_\theta(G)/G]$ is is known to be Deligne-Mumford. Our proof uses the combinatorial structure of the weights of irreducible representations of semisimple groups. As an application we generalize the Grassmannian flop example of Donovan-Segal.
\end{abstract}

\maketitle
% \setcounter{tocdepth}{1}
% \tableofcontents

% TODOS:
% \begin{itemize}
% \item examples
% \item clean up GIT section
% \item degeneracies for type A
% \item pictures in the pieces prop
% \item intro
% \end{itemize}

% \Rachel{PROBLEM: stabilizers of reductive gropu acttions don't have to be reductive. \url{https://cnrs.hal.science/hal-03864542v1/file/unipotent_nov17.pdf} also example is that parabolics stabilize flags

% idea: use other characterization of stability, that stabilizer of $\hat x$ is finite. I think this stabilizer is $stab(x) \cap \ker \theta$, so we have $0 \to finite \to stab(x) \to \Gm$. If the image of stab(x) is finite, then stab(x) is finite and the quotient is DM here. if the image is infinite, I think this means stab(x) contains a torus by \url{https://math.stackexchange.com/questions/4768362/surjective-image-of-maximal-torus-is}???}

\section{Introduction}
\subsection{Weyl-genericity}
Let $V$ be a linear representation of a complex reductive group $G$.
A choice of character $\theta: G \to \Gm$ determines a $G$-invariant semistable locus $V^{ss}_\theta(G) \subseteq V$.
Varieties or Deligne-Mumford stacks that can be written as $[V^{ss}_\theta(G)/G]$ have furnished important examples in geometry and physics: for example, toric varieties and type-A flag varieties can be written this way. Some computations have been performed for general input data $(V, G, \theta)$, but always under a genericity assumption on $\theta$ that ensures the quotient stack $[V^{ss}_\theta(G)/G]$ has finite stabilizer groups. The purpose of this paper is to investigate when these genericity assumptions hold. In other words, we ask: How generic is genericity?

The specific genericity condition we consider is as follows. Let $T \subseteq G$ be a maximal torus with character lattice $\Char(T)$. If we view $V$ as a $T$-representation, there is a locus $\Sigma(V, T) \subseteq \Char(T)_\QQ$ where $V^{ss}_\theta(T)$ is not empty, and a locus $\omega(V, T) \subseteq \Sigma(V, T)$ where $V^{ss}_\theta(T)$ has a point with positive dimensional stabilizer. The Weyl group of $T$ in $G$ acts on $\Char(T)_\QQ$ and we make the following definition.

\begin{definition}
The representation $(V, G)$ is \emph{Weyl-generic} if the Weyl-invariant subspace $\Sigma(V, T)^W$ is not contained in $\omega(V, T).$
\end{definition}
Our main results contain (in different contexts) necessary conditions, sufficient conditions, and equivalent conditions for Weyl-genericity to hold. They include but are not limited to the following.
\begin{itemize}
\item Irreducible Weyl-generic representations are completely classified: roughly speaking, they are built from standard representations of $SL(n)$ and their duals (Corollary \ref{cor:tensor}).
%\item A sufficient condition for Weyl-genericity is given that is somewhat orthogonal to previously known examples (Theorem \ref{thm:sufficient}).
\item In general, if $(V, G)$ is Weyl-generic, then $G$ has type $A$ (Theorem \ref{thm:necessary}).
\end{itemize}
These results suggest that Weyl-genericity is rather special!
\subsection{Representations of semisimple groups}
If $V$ is a representation of a semisimple group $G$, then the character lattice $\Char(G)$ is zero-dimensional and $V$ is never Weyl-generic. However, every complex reductive group $G$ has a finite cover by a group $H \times D$ where $H$ is semisimple and $D$ is a torus, and the structure of $V$ as an $H$ representation is closely tied to whether $V$ is Weyl-generic. The important notion turns out to be the \textit{degeneracy} of $V$ as an $H$-representation, a notion that we define as follows. Let $T \subseteq H$ be a maximal torus.

\begin{definition}[Definition \ref{def:degeni} and Lemma \ref{lem:degen}]
The \emph{degeneracy} of $V$ is the minimum dimension of a cone generated by weights of $V$ that is a subspace of $\Char(T)_\QQ$.
%a cone generated by weights of $V$ that fails to be strongly convex. \Riku{(is this a standard definition?)}
\end{definition}

The degeneracy of $V$ is an integral invariant clearly bounded by the rank of $H$. We say $V$ is \textit{nondegenerate} if equality holds. We classify the nondegenerate representations of the semisimple groups $A_n, B_n, C_n, D_n, E_6, E_7, E_8, F_4$, and $G_2$ in Section 4. Surprisingly, we find that the only such representations are the standard representations of $A_n = SL(n+1)$ and their duals, and the representations $Sym^{2\ell +1}(\CC^2)$ of $SL(2)$. This fact is the fundamental reason for the scarcity of Weyl-generic representations.
We classify the nondegenerate representations of all semisimple groups in Proposition \ref{prop:pieces}.

\subsection{Applications}
The main application of our results is to give examples where theorems in the literature can be applied. For instance, in \cite{HL-Sam} the authors produce derived equivalences from quasi-symmetric representations $(V, G)$. A hypothesis of their main theorem \cite[Thm 1.2]{HL-Sam} is that $G$ acts on $V$ with finite kernel and that the representation is Weyl-generic. 
Similarly, the quasimap theory developed in \cite{toric-qmaps}, \cite{stable-qmaps}, and \cite{orb-qmaps} is a generalization of Gromov-Witten theory that defines numerical invariants for stacks of the form $[V^{ss}_\theta(G)/G]$---provided the stack is Deligne-Mumford. 

It is known that the quotient $[V^{ss}_\theta(G)/G]$ will be Deligne-Mumford for generic choices of $\theta$ if $(V, G)$ is Weyl-generic. Conversely, our Examples \ref{ex:grassmannian} and \ref{ex:quiver1} show that the hypothesis that $[V^{ss}_\theta(G)/G]$ is Deligne-Mumford for some $\theta$ is a priori a little stronger than the Weyl-genericity assumption (see Corollary \ref{cor:oplus}). It turns out that this failure of the converse can be avoided if the dimension of the unstable locus is sufficiently small, relative to the dimensions of both $G$ and $V$, and we prove the following.

\begin{proposition}[Corollary \ref{cor:oplus}]\label{P:intro}
If $(V, G)$ is a Weyl-generic representation, then for $r > \dim G$ there exists $\theta \in \Char(G)$ for which the stack quotient $[(V^{\oplus r})^{ss}_\theta(G)/G]$ is a nonempty Deligne-Mumford stack.
\end{proposition}

Readers interested in constructing Deligne-Mumford stacks will then ask for a list of Weyl-generic representations. Our most general sufficient condition is the following.
\begin{theorem}[Theorem {\ref{thm:sufficient}}]\label{T:intro}
Let $G=H\times D$ be a product of a semisimple group and a torus. Let $X_+$ be a nondegenerate representation of $H$, let $Y$ be a representation of $D$, and let $Z$ be a representation of $G$. Assume that
\begin{itemize}
\item The cone of weights of $Y$ has full dimension.
\item There is a vector $\nu \in \chi(D)_\QQ^\vee$ such that 
\[
\langle \nu, \alpha^+ \rangle \geq 0 \quad \quad \quad \quad \langle \nu, \alpha^-\rangle < 0
\]
for all $D$-weights $\alpha^+$ of $Y$ and $\alpha^-$ of $Z$.
\end{itemize}
  For any integer $t>0$, let $Z^{(t)}$ denote the representation of $G$ obtained by scaling $D$-weights of $Z$ by $t$. 
    Then, for all $t\gg0$, the representation $(X_+\otimes Y) \oplus Z^{(t)} $ is Weyl-generic.
\end{theorem}

Combined with Proposition \ref{P:intro}, this gives many examples of Deligne-Mumford stacks that can be represented as GIT quotients, explicit up to the calculation of an appropriate $t$ appearing in Theorem \ref{T:intro}. For a given $(V, G)$, a lower bound for $t$ such that $(X_+ \otimes Y) \oplus Z^{(t)}$ can in theory be computed explicitly, as we do in Example \ref{ex:sufficient}. As a more explicit result we prove the following, generalizing the Grassmannian flop construction of \cite{DS14}.
\begin{proposition}[Example \ref{ex:gr-flop}]
Let $H = \prod_{i=1}^M SL(n_i)$ and let $X$ be a nondegenerate representation of $H$. Let $\CC_a$ denote the 1-dimensional representation of $\Gm$ of weight $a$. Then the self dual representation
\[
V := (X \otimes \CC_a) \oplus (X^\vee \otimes \CC_{-a})
\]
is Weyl generic.
\end{proposition}

We remark that a large family of Weyl-generic representations of nonabelian groups is already known: these are certain representations arising from moduli of quiver representations. In fact, a quiver-theoretic criterion for existence of $\theta$ such that $[V^{ss}_\theta(G)/G]$ is Deligne-Mumford is give in \cite{FPW}, along with an effective algorithm for computing such $\theta$ when they exist. However, these quiver representations almost never satsify the sufficient conditions of Theorem \ref{T:intro} (see Example \ref{ex:quiver}). Thus our examples provided here are mostly disjoint from those in the literature.

%One might be tempted the think that these kinds of representations are Weyl-generic, but in fact . . . \Rachel{include a sentence like this?}
% \subsection{Outline of the paper}
% Section 2 contains background: we briefly review some of the structure of reductive groups and representations, and then present several variations of Carath\'eodory's theorem. In fact, Carath\'eodory's theorem is probably the most-used tool in our proofs.

% In section 3 we discuss the relationship between Weyl-genericity and the condition that there exists $\theta$ for which $[V^{ss}_\theta(G)/G]$ is Deligne-Mumford. 

% Section 4 

\subsection{Acknowledgements}
This project began as part of the 2022 Michigan Research Experience for Graduates (MREG) and the authors are grateful to the organizers of that event.
%\Rachel{add grant info}

Kurama was partially supported by Murata overseas scholarship and NSF grant DMS-2052750. Li was partially funded by the Simons Collaboration on Perfection in Algebra, Geometry, and Topology.
Talbott was supported by the National Science Foundation Graduate Research Fellowship under Grant No. DGE 2241144.

Webb was supported by the NSF grant DMS 2501528 and an NSF Postdoctoral Research Fellowship, award number 200213, as well as a grant from the Simons Foundation. Webb is also grateful to the Isaac Newton Institute for Mathematical Sciences, Cambridge, for support and hospitality during the program ``New Equivariant Methods in Algebraic and Differential Geometry,'' where ideas in this paper developed significantly.

The authors thank Kimoi Kemboi, David Speyer, Richard Thomas, and Daniel Halpern-Leistner for helpful discussions.

\section{Background}
\subsection{Reductive groups and their representations}

We review some structure theorems for reductive groups and their representations. The results in this section are not new. Throughout, $G$ is a connected complex reductive group.
\begin{theorem}\label{thm:structure}
Let $G$ be a connected complex reductive group. Then there is a surjective homomorphism with finite central kernel
\[
f: H_1 \times \ldots \times H_M \times D \to G
\]
where each $H_i$ is simply connected and almost-simple 
%(\Riku{is this a standard definition?}) 
and $D$ is the maximal central torus of $G$. Moreover the preimage of a maximal torus $K$ of $G$ is $T \times D$ where $T$ is a maximal torus of $\prod_i H_i=:H$, the Weyl group of $H \times D$ is isomorphic to the Weyl group of $G$, and $f$ induces a Weyl-equivariant isomorphism
\[
\Char(K)_\QQ \to \Char(T \times D)_\QQ.
\]
\end{theorem}
\begin{proof}
If $\sD(G)$ is the derived subgroup and $D$ is a maximal central torus of $G$, there is a central isogeny $\sD(G) \times D \to G$ (see e.g. \cite[Thm 3.2.2]{Conrad} 
%\Rachel{BETTER: 5.3.3 \url{https://math.stanford.edu/~conrad/papers/luminysga3.pdf}}
). Since the derived subgroup $\sD(G)$ is semisimple, there is a product of simply connected almost-simple groups $H := \prod_{i=1}^M H_i$ and a central isogeny $\prod_{i=1}^M H_i \to \sD(G)$ \cite[17.27]{Milne}. The map $f$ is obtained from the composition of these two central isogenies.

Let $T$ denote a maximal torus of $H$ and let $K = f(H \times T)$. Then $T \times D$ is a maximal torus of $H \times D$ and $K$ is a maximal torus of $G$ (see e.g. \cite[Cor 21.3C]{Humphreys}). Since
\[
\ker(T \times D \to K) = \ker(f)
\]
is finite, the cokernel of the dual homomorphism
\begin{equation}\label{eq:chars}
\Char(K) \to \Char(T \times D)
\end{equation}
is also finite and hence \eqref{eq:chars} becomes an isomorphism after tensoring with $\QQ$.

Let $W(K, G)$ (resp. $W(T \times D, H \times D)$) denote the Weyl group of $K$ in $G$ (resp. of $T \times D$ in $H \times D$). By \cite[Prop 21.4B]{Humphreys} the map $W(T \times D, H \times D) \to W(K, G) $ induced by $f$ is an isomorphism. It is straightforward to check that \eqref{eq:chars} is equivariant under this identification.
\end{proof}

The simply connected almost-simple groups are in bijection with connected Dynkin diagrams (see e.g. \cite[19.64]{Milne}). We denote such a group by its associated Dynkin diagram; hence, the simply connected almost-simple groups are denoted 
\[A_n (n \geq 1), \;\;B_n (n \geq 2),\;\; C_n (n \geq 3), \;\;D_n (n \geq 4), \;\;E_6, \;\;E_7,\;\; E_8, \;\;F_4, \;\text{ and }\; G_2.\] 
The subscript is always the rank of the group.

\begin{example}\label{ex:groupsnew}
The group $A_n$ is just $SL({n+1})$ (for $n \geq 1$) and $C(n)$ is $Sp({2n})$ (for $n \geq 3$). The groups $B(n)$ and $D(n)$ are the simply connected covers of $SO({2n+1})$ (for $n \geq 2$) and $SO({2n})$ (for $n \geq 4$, respectively. 
\end{example}
\begin{definition}
A connected complex reductive group \emph{has type $A$} if it is isogeneous to a group $H_1 \times \ldots \times H_M \times D$ where $D$ is a torus and each $H_i$ is isomorphic to $SL({n_i})$ for some $n_i \geq 2$. 
\end{definition}

Let $H$ be a simply connected almost-simple group and let $T \subseteq H$ be a maximal torus. Recall that a choice of Borel subgroup $B \subseteq H$ containing $T$ determines a set of \emph{positive roots} $\Phi^+ \subseteq \Char(T)$.
Moreover for each $\alpha \in \Phi^+$ there is a 1-parameter subgroup $\alpha^\vee$ of $T$ with the property that $\langle \alpha, \alpha^\vee \rangle = 2$, where $\langle-,-\rangle$ is the canonical pairing between characters and 1-parameter subgroups of $T$.
A weight $\lambda \in \Char(T)$ is \emph{dominant} if $\langle \lambda, \alpha^\vee \rangle \geq 0$ for all $\alpha \in \Phi^+$.
There is also a \emph{dominance order} on $\Char(T)$, where for $\lambda, \mu \in \Char(T)$ we say $\lambda \leq \mu$ if $\mu-\lambda$ can be written as a nonnegative integral sum of positive roots. 

A representation of $H$ is \textit{simple} if it is nonzero and it has no proper nonzero subrepresentations. By the theorem of highest weight, simple representations of $H$ biject with dominant weights in $\chi(T)$. Using this theorem, the dominance order on the set of dominant weights has an equivalent representation-theoretic description:

\begin{lemma}[{\cite[Cor 1.10, Rem 1.11]{Stembridge}}] \label{lem:weights-contained} Let $\lambda, \mu \in \chi(T)$ be dominant weights. Then $\lambda \leq \mu$ if and only if the weights of the representation corresponding to $\lambda$ are a subset of the weights of the representation corresponding to $\mu$.
\end{lemma}
\begin{remark}\label{rmk:finding-weights}
If $\lambda$ is a dominant weight, the $T$-weights of $V_\lambda$ are those weights $\mu$ 
% (\Riku{Maybe i'm misremembering but don't we need to intersect with an appropriate lattice here?} \Rachel{The weight lattice for a semisimple group with maximal torus $T$ is $\Char(T)$.})
such that some element of their Weyl-orbit is dominated by $\lambda$ (see e.g. \cite[Rem 1.11]{Stembridge}). 
\end{remark}

In the next section we will analyze representations of $H$ that correspond to dominant weights $\lambda$ that are minimal for the order $\leq$. The representations have the property that the weights of an arbitrary representation of $H$ will contain the weights of at least one of these dominance-minimal weights as a subset. 
\begin{lemma}[{\cite[Prop 1.12]{Stembridge}}]
A dominant weight is minimal with respect to $\leq$ if and only if it is 0 or \emph{minuscule}.
\end{lemma} 
For us, the lemma functions as the definition of a minuscule weight (i.e., a dominant weight $\lambda \in \chi(T)$ is minuscule if it is dominance minimal and nonzero). The minuscule weights of each simply connected almost-simple group are finite in number and known explicitly: they are listed, for example, in \cite[Chapter VI, Exercise 4.15 (p.232)]{Bourbaki}.

\subsection{Convex geometry}
We recall Carath\'eodory's theorem and some consequences. The only new result in this section is Proposition \ref{P:box-new}, an application of Carath\'eodory's theorem that we will use extensively.

If $I = \{\xi_1, \ldots, \xi_n\}$ is a finite subset of a rational vector space $V$, then the cone generated by $I$ is the set
\[
\Cone(I) := \{ \sum_{\xi \in I} a_\xi \xi \mid a_\xi \in \QQ_{\geq 0}\}.
\]

\begin{lemma}[Carath\'eodory]\label{lem:car}
Let $I$ be a finite subset of a rational vector space. If $\Cone(I)$ spans a linear subspace of dimension at most $m$ and $\theta \in \Cone(I)$, then $\theta$ is in $\Cone(J)$ for some $J \subset I$ of cardinality at most $m.$
\end{lemma}
\begin{proof}
Suppose $n$ is the minimal positive integer such that $\theta = \sum_{i=1}^n a_i \xi_i$ for some $c_i \in \QQ_{\geq 0}$ and $\xi_1, \ldots, \xi_n \in I$. If $n > m$ (arguing by contradiction), then there are rational numbers $b_1, \ldots, b_n$, not all zero, such that $\sum_{i=1}^n b_i\xi_i = 0.$  It follows that
\[
\theta = \sum (a_i + \epsilon b_i) \xi_i
\]
for all rational $\epsilon$. We will show that we can choose $\epsilon$ such that all coefficients $a_i + \epsilon b_i$ are nonnegative and at least one of them is zero. This contradicts minimality of $n$.

To choose $\epsilon$, note first that when $\epsilon$ is sufficiently small all coefficients $a_i + \epsilon b_i$ are nonnegative (since $a_i$ is positive). Next, for each $i$ there is at most one $\epsilon_i$ such that $a_i + \epsilon_i b_i=0$. If we set $\epsilon_i=\infty$ when no such finite $\epsilon$ exists, then for all $i$ it is true that for $|\epsilon| < |\epsilon_i|$ we have $a_i + \epsilon b_i > 0$.  Finally, if we choose $\epsilon$ to equal the minimum of the finite set $\{1, \epsilon_1, \ldots, \epsilon_n\}$, then $\epsilon$ satisfies all the required conditions.
\end{proof}

\begin{corollary}\label{cor:car}
Let $I$ be a nonempty finite subset of a rational vector space and let $m$ be the dimension of the span of $I$. If there exist $a_\xi \in \QQ_{\geq 0}$ for $ \xi \in I$ satisfying
\begin{equation}\label{eq:degency}
\sum_{\xi \in I} a_\xi \xi = 0 \quad \quad \quad \sum_{\xi \in I} a_\xi = 1
\end{equation}
then one can arrange for \eqref{eq:degency} to hold with at most $m+1$ of the $a_\xi$ nonzero.
\end{corollary}
\begin{proof}
%If 0 is in $I$ then the statement  is clearly true. So assume $\xi \neq 0$ for every $\xi \in I$.

Given an expression \eqref{eq:degency}, choose $\xi_1 \in I$ such that $a_{\xi_1} \neq 0$. Then 
\[
-a_1\xi_1 = \sum_{\xi \in I \setminus\{\xi_1\}} a_\xi \xi
\]
and by Lemma \ref{lem:car}, the right hand side is in $\Cone(J)$ for some $J \subseteq (I \setminus \{\xi_1\})$ of cardinality at most $m$. So we may write
\[
0 = a_1\xi_1 + \sum_{\xi \in J} a_\xi \xi
\]
with $a_1, a_\xi \in \QQ_{\geq 0}$ and $a_1 + \sum_{\xi \in J} a_\xi > 0$. The corollary follows.
\end{proof}

% \Rachel{I made a small correction to the proof of the following. Please check!}
% \Riku{I agree about this correction thanks!}

%\Rachel{just added the requirement that $n, m \geq 1$. I think the conclusion is false without it. also added the case $v \neq 0.$}
\begin{proposition}\label{P:box-new}
Let $V$ and $W$ be vector spaces over $\QQ$. Let $v \in V$ and $w \in W$ be vectors, let $m$ and $n$ be positive integers, and let $v_1,..,v_n\in V,$ and $w_1,..,w_m\in W$ be vectors satisfying 
\[\sum_{i=1}^n a_iv_i=v \quad \quad \quad \text{and} \quad \quad \quad \sum_{j=1}^m b_jw_j=w\]
for some \emph{positive} rational numbers $a_i$ and $b_j$. If $v\neq 0$ and $w\neq 0$ then furthermore assume $\sum_{i=1}^n a_i = \sum_{j=1}^m b_j$. Then 
there is an equality 
\[\sum_{k=1}^{n+m-1} c_k(v_{i_k},w_{j_k})=(v, w)\] for some $c_k\in\QQ_{\geq0}$ and at least one $c_k$ positive.
% Let $V$ and $W$ be vector spaces over $\QQ$. Let $w \in W$ be a vector, let $m$ and $n$ be positive integers, and let $v_1,..,v_n\in V,$ and $w_1,..,w_m\in W$ be vectors satisfying 
% \[\sum_{i=1}^n a_iv_i=0 \quad \quad \quad \text{and} \quad \quad \quad \sum_{j=1}^m b_jw_j=w\]
% for some \emph{positive} rational numbers $a_i$ and $b_j$. Then 
% there is an equality $\sum_{k=1}^{n+m-1} c_k(v_{i_k},w_{j_k})=(0, w)$ for some $c_k\in\QQ_{\geq0}$ and at least one $c_k$ positive.
\end{proposition}
We give two proofs of Proposition \ref{P:box-new}, one using Carath\'eodory's theorem (which is not constructive), and one that gives an explicit algorithm for finding a set of coefficients $c_k$.
\begin{proof}[First proof of \ref{P:box-new}]

Let $|a| = \sum_{i=1}^n a_i$ and let $|b| = \sum_{j=1}^m b_j$. We note that if $v = 0$, we may replace $a_i$ by $|a|^{-1}|b|a_i$ to ensure that $|a|=|b|$ holds. We can make an analogous replacement if $w=0$ and thus assume $|a| = |b|$ holds in general.
Therefore it is enough to prove the proposition in the universal 
situation where $\{v_i\}_{i=1}^n$ is a basis for $V$ and $\{w_j\}_{j=1}^m$ is a basis for $W$ and $|a| = |b|$. 
From the universal result the special result is obtained by projecting these bases to the original (potentially linearly dependent) vectors $v_i$ and $w_j$.

In 
this 
%the universal 
setting, we have $\dim V = n$ and $\dim W = m$, so $\dim V \times W = n + m$. Let $v^*$ be the linear function that sends $\sum_{i=1}^n d_i v_i$ to $|d|$ and let $w^*$ be the linear function that sends $\sum_{j=1}^m e_jw_j$ to $|e|$. Then the vectors $(v_i, w_j)$ and $(v, w)$  live in $\ker(v^*-w^*)$ which is a vector space of dimension $n+m-1$. Carath\'eodory's theorem \ref{lem:car} implies that $(v, w)$ can be realized as a nonnegative rational linear combination of at most $n+m-1$ of the vectors $(v_i, w_j)$. At least one of these coefficients is positive since $v$ is nonzero (being the sum of elements of a basis of a positive-dimensional vector space).

% Without loss of generality we may assume $V$ is spanned by the $v_i$ and $W$ is spanned by the $w_i$, so $\dim V \leq n-1$ and $\dim W \leq m$. The vectors $(v_i, w_j)$ live in a vector space $V \times W$ of dimension at most $n+m-1$. Let $|a|$ be the sum $\sum_{i=1}^n a_i$ and let $|b|$ be the sum $\sum_{j=1}^m b_j$. We have that
% \begin{equation}\label{eq:box1}
% \sum_{i=1}^n\sum_{j=1}^m |a|^{-1}a_ib_j(v_i, w_j) = \Big(|a|^{-1}|b|\sum_i  a_iv_i, \;|a|^{-1}|a|\sum_j b_jw_j\Big) = (0, w).
% \end{equation}
% Now Caratheodory's theorem \ref{lem:car} implies that $(0, w)$ can be realized as a nonnegative linear combination of at most $n+m-1$ of the vectors $(v_i, w_j)$. 

% If $w \neq 0$ then at least one of the $c_k$ is positive and this completes the proof of the proposition.
% If $w=0$ then $\dim W \leq m-1$ and $\dim V \times W \leq n+m-2$. The proposition in this case follows from \eqref{eq:box1} and Corollary \ref{cor:car}.
\end{proof}

\begin{proof}[Second proof of \ref{P:box-new}]
We give an algorithm that constructs the required coefficients $c_{ij}$. Let $|a| = \sum_i a_i$ and $|b| = \sum_j b_j$. As in the previous proof, we may assume $|a| = |b|.$

To set up the algorithm, let $C=(c_{ij})$ be the $n \times m$ matrix with all entries equal to zero. Our goal is to iteratively modify entries of $C$ until
\begin{enumerate}
\item[(i)] $c_{n m} \neq 0$ 
\item[(ii)] $\sum_j c_{ij} = a_i,$ for all $i$ and $\sum_i c_{ij} = b_j$ for all $j$
\item[(iii)] at most $n+m-1$ of the $c_{ij}$ are nonzero
\item[(iv)] all $c_{ij}$ are nonnegative.
\end{enumerate}
Condition (ii) implies that $\sum_{i, j} c_{ij}(v_{i}, w_{j}) = (v, w)$, so at the termination of the algorithm the entries of $C$ will be the desired coefficients.
% Clearly, condition (ii) will be impossible unless 
% \begin{equation}\label{eq:box2}
% |a| = |b|.
% \end{equation} 
% If $\sum_i a_i v_i = 0$, we may replace $a_i$ by $|a|^{-1}|b|a_i$ to ensure that \eqref{eq:box2} holds. We can make an analogous replacement if $w=0$.

The algorithm runs as follows.
\begin{enumerate}
\item Initiate a location coordinate $(i,j)$ to the value $(1,1)$.
\item Update $c_{ij}$: If the location coordinate has value $(i, j)$, set $c_{ij}$ to be the minimum of the two quantities
\begin{equation}\label{eq:two}
\; a_i - \sum_{k < j} c_{ik}\;\quad\quad\quad \text{and}\quad\quad\quad b_j - \sum_{k < i} c_{kj} \;.
\end{equation}
\item Update the location coordinate: If the first quantity in \eqref{eq:two} equals minimum, increase $i$ by one. If the second quantity equals the minimum, increase $j$ by one. 
%(Both may occur.)
\item Repeat steps 2-3. Stop when $i>n$ or $j>m$.
\end{enumerate}

When the algorithm terminates, our matrix $C$ will satisfy (i) (since $a_n$ and $b_m$ are positive). It will also satisfy (iii), as an inductive argument
%(\Riku{maybe it's good to explain this argument?}) 
shows that after updating entry $c_{ij}$ at most $i+j-1$ entries of the matrix are nonzero. (For the inductive step, note that each iteration of steps 2-3 adds at most 1 to the number of nonzero entries in $C$ and it adds exactly 1 to the quantity $i+j-1$.)

To show that (ii) and (iv) hold, note first that at each iteration of the algorithm, for all $x=1, \ldots, n$ and $y = 1, \ldots, m$ we have
\[
\sum_{k}c_{xk} \leq a_x \quad\quad\quad \text{and}\quad\quad\quad \sum_k c_{ky} \leq b_y
\]
with equality if $x < i$ or $y < j$. The inequalities follow from always choosing $c_{ij}$ to be the \textit{minimum} of the quantities \eqref{eq:two}, and the equalities are a consequence of our rule for updating the location coordinate. From this (iv) follows from our formula for $c_{ij}$.

% \Riku{i'm a little confused here, since $n+1$ is not a valid index for the matrix}
% \Rachel{One way for the algorithm to terminate is to have $i>n$. The pair $(n+1, j)$ is one case of what the location coordinate can look like when the algorithm stops.}
We now show that the algorithm terminates after filling in $c_{nm}$. For the first claim, suppose the algorithm terminates when the location counter is $(n+1, j)$ for some $j<m$. Then since the $i$th row sums to $a_i$ for all $i=1, \ldots n$, the entries in $C$ add to $|a| = |b|$. On the other hand, the sum of the entries in the $j$th column is at most $b_j$ and the sum of the $m$th column is $0 < b_m$. This is a contradiction. An analogous argument shows that the algorithm cannot terminate at $(i, m+1)$ for $i<n$.

Hence the algorithm terminates when the location coordinate is $(i, j)$ for $i \geq n$ and $j \geq m$ and at least one of these equalities strict. In light of the previous discussion, to show (ii) it suffices to check that in the final step of the algorithm, the two quantities
\[
 a_n - \sum_{k < m} c_{nk}\;\quad\quad\quad \text{and}\quad\quad\quad b_m - \sum_{k < n} c_{km} \;.
\]
are equal. But this holds because both are equal to the difference $|a| - |C| = |b| - |C|$, where $|C|$ is the sum $\sum_{(i, j) \neq (n, m)} c_{ij}$.

\end{proof}

\section{Twisted affine GIT}
Much of this section is expository; only the final Example \ref{ex:quiver1}, Proposition \ref{prop:oplus}, and Corollary \ref{cor:oplus} are new.

\subsection{Affine GIT}
Let $G$ be a complex reductive group. A \textit{1-parameter subgroup} of $G$ is a homomorphism $\Gm \to G$, and a \textit{character} of $G$ is a homomorphism $ G \to \Gm$. If $\theta$ is a character and $\lambda$ is a 1-parameter subgroup of $G$, we have an integer-valued pairing
\begin{equation}\label{eq:pairing}
\langle \theta, \lambda \rangle = a \quad \quad \quad \text{if} \quad \quad \quad \lambda(\theta(t)) = t^a.
\end{equation}

The data of a twisted affine GIT quotient is a triple $(V, G, \theta)$ where $V$ is a linear representation of $G$ and $\theta$ is a character of $G$. The associated locus of semistable points was defined in \cite{GIT} and may be characterized by the following numerical criterion of \cite[Prop 2.5]{king}.

\begin{definition}\label{def:ss}
The \emph{{semistable locus}} $V^{\sst}_{\theta}(G) \subseteq V$ is the open subvariety whose $\CC$-points are $x\in V(\CC)$ such that whenever $\lim_{t\to 0}\lambda(t)\cdot x$ exists for a $1$-parameter subgroup $ \lambda$ of $G$, we have $\langle \theta, \lambda \rangle\geq 0$.
\end{definition}
The \emph{unstable locus} is the complement of the semistable locus and denoted $V^{us}_\theta(G)$.

When $T \subseteq G$ is a maximal torus, the following lemma relates semistability for $T$ and $G$.

\begin{lemma}\label{lem:torus-to-group}

Let $(V,G,\theta)$ be as above and let $T \subseteq G$ be a maximal torus. Then
\[V^{ss}_\theta(G)=\bigcap_{g\in G}gV_\theta^{ss}(T) \quad \quad \text{and hence}\quad\quad V^{us}_\theta(G) = G \,\cdot\,V^{us}_\theta(T).\]
\end{lemma}

\begin{proof}
We have $x\in \bigcap_{g\in G}gV_\theta^{ss}(T)$ if and only if, for all $g \in G$ and $\lambda: \Gm \to T$, the existence of the limit $\lim_{t \to 0} \lambda(t)g^{-1}x$ implies $\langle \theta, \lambda \rangle \geq 0$. This limit exists if and only if $\lim_{t \to 0} g\lambda(t)g^{-1}x$ exists. Moreover all 1-parameter subgroups of $G$ are conjugate to 1-parameter subgroups of $T$, so we have $x\in \bigcap_{g\in G}gV_\theta^{ss}(T)$ precisely when $x \in V^{ss}_\theta(G)$.

%Assume $x\in \bigcap_{g\in G}gV_\theta^{ss}(T)$. For any $1$-parameter subgroup $\lambda(t)$ of $G$ such that $\lim_{t\to 0}\lambda(t)\cdot x$ exists, there exists a $g\in G$ such that $g\lambda(t)g^{-1}$ is a $1$-parameter subgroup of $T$. By continuity of the $g$-action, $\lim_{t\to 0}g\lambda(t)g^{-1}\cdot (gx)$ exists, so $\langle \theta, g\lambda g^{-1} \rangle =\langle\theta, \lambda\rangle\geq 0$ and $x\in V_\theta^{ss}(G)$.
%Conversely if $x\in V^{ss}_\theta(G)$, and fix a $g\in G$. For any $1$-parameter subgroup $\lambda(t)$ of $T$, if $\lim_{t\to 0}g\lambda(t)g^{-1}\cdot (gx)$ exists, so does $\lim_{t\to 0}\lambda(t) x$, and so $\langle \theta, \lambda \rangle=\langle \theta, g\lambda g^{-1}\rangle\geq 0$. As a consequence, $gx\in V_\theta^{ss}(T)$ for all $g \in G$. 
\end{proof}

\subsection{Affine VGIT}Let $G$ be a complex reductive group and let $V$ be a representation of $G$.
Let $\Char(G)$ denote the character group of $G$, and set $\Char(G)_\QQ := \Char(G) \otimes \QQ$. The definition of the pairing \eqref{eq:pairing} between characters and 1-parameter subgroups extends linearly to a $\QQ$-valued pairing between elements of $\Char(G)_\QQ$ and 1-parameter subgroups of $G$. With this, the definition of $V^{ss}_\theta(G)$ in terms of the numerical criterion makes sense for $\theta \in \Char(G)_\QQ$.

The \emph{semistable cone} of $(V, G)$ is the set
\[\Sigma(V, G) := \{ \theta \in \Char(G)_\QQ \mid V^{ss}_\theta(G) \neq \emptyset\}\]
and the \emph{walls} of $(V, G)$ are the set 
\[
\omega(V, G) := \{\theta \in \Sigma(V, G) \mid \text{some }x \in V^{ss}_\theta(G) \text{ has positive dimensional stabilizer}\}.
\]
The set $\omega(V, G)$ will contain the boundary of $\Sigma(V, G)$ but can also contain interior points, and may even contain a cone of dimension equal to the dimension of $\Sigma(V, G)$.

\begin{remark}
If $G$ acts on $V$ with finite kernel, the walls of $(V, G)$ are also the locus of $\theta \in \Sigma(V, G)$ such that $V^{ss}_\theta(G)$ has a strictly semistable point (in the sense of GIT).
\end{remark}

If $G$ is abelian, the sets $\Sigma(V, G)$ and $\omega(V, G)$ can be written down explicitly. Let $n$ be the dimension of $V$ and fix a weight basis for the $G$-action, so the action of $g \in G$ on $(x_1, \ldots, x_n) \in V$ can be written as
\[
g \cdot (x_1, \ldots, x_n) = (\xi_1(g) x_1, \ldots, \xi_n(g) x_n)
\]
for some characters $\xi_1, \ldots, \xi_n$ of $G$. For $x \in V$, define 
\[\supp(x) := \{i \in \{1, \ldots, n\}  \mid x_i \neq 0\}.\]
The following lemma is well-known.

%\Rachel{Riku wondered whether $\omega(V, G)$ was supposed to be a union over sets $I$ of size $\leq k-1$. I think this is equivalent to what currently appears, and I prefer the way it currently appears because there are fewer things in the union.}
\begin{lemma}\label{lem:toric}
Let $G = \Gm^k$ act on $V$ with weights $\xi_1, \ldots, \xi_n$. Then
for any $\theta \in \Char(G)_\QQ$, we have $V^{ss}_\theta(G) = \{x \in V \mid \theta \in \Cone(\{\xi_i\}_{i \in \supp(x)}) \}.$
Moreover, we have equalities 
\[\Sigma(V, G) = \Cone(\{\xi_i\}_{i=1}^n) \quad \quad \quad\text{and} \quad \quad \quad \omega(V, G) = \bigcup_{\substack{I \subseteq \{1, \ldots, n\}\\|I|\leq k-1 }} \Cone(\{\xi_i\}_{i \in I}).\]
\end{lemma}
\begin{proof}
For $x \in V$ let $C_x \subset \chi(G)_\QQ$ denote the cone $\Cone(\{\xi_i\}_{i \in \supp(x)})$ and let $C_x^\vee$ denote its dual.
If $\lambda: \Gm \to G$ is a 1-parameter subgroup, then the following are equivalent:
\begin{enumerate}[label=(\alph*)]
\item $\lim_{t \to 0} \lambda(t) x$ exists.
\item $\langle \xi_i, \lambda \rangle \geq 0$ for all $i \in \supp(x)$.
\item $\lambda \in C_x^\vee$.
\end{enumerate}
It follows that $x$ is in $V^{ss}_\theta(G)$ if and only if $C_x^\vee \subseteq \Cone(\{\theta\})^\vee$, where we note that $\Cone(\{\theta\})$ is a halfspace determined by $\theta$. But this is equivalent to $\theta \in C_x$, proving the formula for $V^{ss}_\theta(G)$.

The formula for $\Sigma(V, G)$ follows immediately from the formula for $V^{ss}_\theta(G)$. To obtain the formula for $\omega(V, G)$, observe that $\theta$ is in $\omega(V, G)$ if and only if there is a pair $(x, \lambda)$ with $x \in V^{ss}_\theta(G)$ and $\lambda$ a 1-parameter subgroup of $G$ with image in the stabilizer of $x$. This is equivalent to requiring 
\begin{equation}\label{eq:partway}
\theta \in C_x \quad \quad \text{and} \quad \quad \langle \xi_i, \lambda \rangle  = 0 \text{ for all } i \in \supp(x).
\end{equation}
Suppose that for some $\theta$ we can find $(x, \lambda)$ such that these conditions hold. Then $C_x$ is contained in the $k-1$-dimensional subspace of $\Char(G)_\QQ$ orthogonal to $\lambda$. By Carath\'eodory's theorem (Lemma \ref{lem:car}) we can find a subset $I \subseteq \supp(x)$ of size %\Riku{$\leq$?} 
at most
$k-1$ such that $\theta \in \Cone(I)$.

Conversely, suppose $\theta \in \Cone(\{\xi_i\}_{i \in I})$ for some $I \subseteq \{1, \ldots, n\}$ of size %\Riku{$\leq$?} 
at most $k-1$. Then define $x \in V$ by $x_i = 1$ if $i \in I$ and $x_i=0$ otherwise, and choose $\lambda$ to be any rational vector in $\Hom(\Gm,G)_{\QQ}$ perpendicular to the subspace spanned by $\Cone(\{\xi_i\}_{i \in I})$. Then $(x, \lambda)$ satisfy the conditions \eqref{eq:partway}.

\end{proof}

\begin{remark}
In Lemma \ref{lem:toric}, if $n$ is greater than or equal to $k$, it is easy to see that in the formula for $\omega(V, G)$ we can replace $|I| \leq k-1$ with $|I|=k$. If $n< k$ this is only true if we allow $I$ to have repeated elements. We define a \emph{wall} for $(V, \Gm^k)$ to be a cone in $\Char(\Gm^k)$ generated by $\xi_1, \ldots, \xi_{k-1}$, where the $\xi_i$ are weights of $V$ (possibly non-distinct).
The union of these individual walls is the set $\omega(V, G)$.
\end{remark}

\begin{proposition}\label{prop:toric}
If $G$ is abelian, the unstable locus of $V$ is given by
\[
V^{us}_\theta(G) = \bigcup_{\lambda \text{ s.t. } \langle \theta, \lambda \rangle < 0} V^{\lambda \geq 0} \quad \quad \text{where} \quad \quad V^{\lambda \geq 0} := \{(x_i)_{i=1}^n \mid x_i=0 \text{ if } \langle \xi_i, \lambda \rangle < 0 \}.
\]
In particular the irreducible components of $V^{us}_\theta(G)$ are each of the form $V^{\lambda \geq 0}$ for some $\lambda$.
% \Riku{
% Moreover, this is a finite union in the sense that there exist $\lambda_i$ ($i=1,2,..,N$) with $\langle\theta,\lambda_i\rangle<0$ such that $V^{us}_\theta(G) = \bigcup_{i=1}^N V^{\lambda_i \geq 0}$.
% }
\end{proposition}
\begin{proof}
The formula for $V^{us}_\theta(G) $ follows from the formula for $V^{ss}_\theta(G)$ in Lemma \ref{lem:toric}. Since there are only finitely many possible subspaces that can equal $V^{\lambda \geq 0}$ for some $\lambda$ (namely the coordinate subspaces of $V$), we can write $V^{us}_\theta(G)$ as a union of finitely many $V^\lambda$, and then the statement about irreducible components follows from \cite[0G2Y]{stacks-project}.
\end{proof}

% \Riku{
% Actually it seems like a proposition than a corollary? This stronger version is used in the dimension count argument.
% \begin{proof}
%     The first sentence is an immediate consequence of the definition of semistability. For the finiteness of union, we observe that the subsets $V^{\lambda \geq 0}$ are determined by choosing a collection of coordinates which are required to vanish. We conclude by the fact that there are only finitely many choices of such collections of coordinates.
% \end{proof}
% }
\subsection{VGIT and abelianization}
Let $G$ be a \textit{connected} complex reductive group, let $T \subseteq G$ be a maximal torus, and let $W$ be the associated Weyl group. Restriction of characters induces a linear map
\begin{equation}\label{eq:embed-characters}
\chi(G)_\QQ \to \chi(T)_\QQ.
\end{equation}

The authors are grateful to Loren Spice \cite{spice} for explaining the proof of the following lemma.

\begin{lemma}\label{lem:spice}
The linear map \eqref{eq:embed-characters} is injective with image $\Char(T)_\QQ^W$.
\end{lemma}
\begin{proof}
The identification $\Char(G)\simeq \Char(T)^W$ essentially follows from an isomorpihsm
\[
T/(T \cap \sD(G)) \to G/\sD(G).
\]
This morphism is clearly injective, and it is surjective because $G$ is generated by the derived subgroup $\sD(G)$ and its center (see e.g. \cite[Thm 3.2.2]{Conrad}). It follows that
\[
\Hom(G, \Gm) = \Hom(G/\sD(G), \Gm) = \Hom(T/(T \cap \sD(G)), \Gm).
\]

It remains to show that $\theta \in \Char(T)$ is Weyl-invariant if and only if it vanishes on $T \cap \sD(G)$. But $\theta$ is Weyl invariant if and only if $s_\alpha(\theta) = \theta$ for all simple roots $\alpha$ (where $s_\alpha$ is the reflection along $\alpha$), if and only if $\langle \theta, \alpha^\vee \rangle = 0$, if and only if the image of the 1-parameter subgroup $\alpha^\vee$ is contained in $\ker(\theta)$. Since $T \cap \sD(G)$ is generated by the images of $\alpha^\vee$ as $\alpha$ ranges over all simple roots (see e.g. \cite[Example 2.1]{Conrad2}) we are done.

\end{proof}

Let $V$ be a representation of $G$. By restriction we have also a representation of $T$.
Since by the lemma $\Char(G)_\QQ  = \Char(T)_\QQ^W$, it is natural to ask whether $\Sigma(V, G) = \Sigma(V, T)^W$ and likewise whether $\omega(V, G) = \omega(V, T)^W$. It is easy to see that containment holds in one direction (see e.g. \cite[Prop 2.17]{Mathieu} or \cite[Prop 2.1]{HL-Sam}).

\begin{lemma}\label{lem:abelianization}
Let $G$ be a connected complex reductive group with maximal torus $T$ and Weyl group $W$, and let $V$ be a representation of $G$. Then
\[
\Sigma(V, G) \subseteq \Sigma(V, T)^W \quad \quad \text{and} \quad \quad \omega(V, G) \subseteq \omega(V, T)^W.
\]
\end{lemma}
\begin{proof}
The first containment is an immediate consequence of Lemma \ref{lem:torus-to-group}. For the second containment, 
suppose $x \in V^{ss}_\theta(G)$ has a positive dimensional stabilizer. The $G$-orbit of $x$ may not be closed in $V^{ss}_\theta(G)$; if not choose $x'$ in the closure of the orbit of $x$ whose orbit is indeed closed in $V^{ss}_\theta(G)$. By upper semicontinuity of fiber dimension the stabilizer of $x'$ is still positive dimensional. 
%See https://mathoverflow.net/questions/193/when-is-fiber-dimension-upper-semi-continuous
Because the orbit of $x'$ is closed in $V^{ss}_\theta(G)$, this point is polystable, hence has a reductive stabilizer (this is a consequence of Matsushima's Criterion). 
%copied from Dan's email: if G is reductive, then H\subset G is reductive if and only if G/H is affine. In this case H is the stabilizer of a point in X^{ss}, and G/H is its orbit. If the orbit G/H is closed in X^{ss}, then it is also closed in a G-equivariant affine open containing the orbit. Hence it is affine.
Since the stabilizer is positive dimensional it contains a nontrivial maximal torus $S_{x'}$. There is some $g \in G$ for which $gS_{x'}g^{-1} \subseteq T$, and $gS_{x'}g^{-1}$ is the stabilizer of $gx'$. But $gx'$ is also in $V^{ss}_\theta(G) $, hence in $V^{ss}_\theta(T)$, and so $\theta$ is in $\omega(V, T)$.
\end{proof}

The converse to this lemma is not true. Our first example is rather na\"ive, but as we will see in Proposition \ref{prop:oplus}, it captures exactly 
how the equality $\Sigma(V, G) \subseteq \Sigma(V, T)^W $ can fail in general.
%what can go wrong in general.

\begin{example}\label{ex:grassmannian}
Let $G = GL(n)$, let $\theta$ be the determinant character, and let $V = \AA^n$ be the standard representation with coordinates $x_1, \ldots, x_n$. Let $T\subseteq G$ be the diagonal matrices. Then $V^{us}_\theta(T)$ is the union of the hyperplanes $x_i=0$ for $i=1, \ldots, n$, and in particular $V^{ss}_\theta(T) \neq \emptyset$. By Lemma \ref{lem:torus-to-group} we have
\[
V^{us}_\theta(G) = G \cdot V^{us}_\theta(T),
\]
but this is all of $V$. So $V^{ss}_\theta(G) = \emptyset$ and 
\[
\Sigma(V, G) \subsetneq \Sigma(V, T)^W.
\]

In general, if $r$ is a positive integer, $(V^{\oplus r})^{us}_\theta(T)$ can be identified with $n \times r$ matrices with one row equal to zero. The group $GL(n)$ acts by left multiplication and $G \cdot (V^{\oplus r})^{us}_\theta(T)$ is the locus of matrices with low row rank. So if $r \geq n$ then $(V^{\oplus r})^{ss}_\theta(T)$ is no longer empty, and in fact the GIT quotient $V^{\oplus r} \sslash_\theta G$ is the Grassmannian variety of $n$-planes in $\AA^r$.

\end{example}

Our second example shows that even the containment 
\[\omega(V, G)\; \subseteq \;\omega(V, T)^W \cap \Sigma(V, G)\]
can be strict.

\begin{example}\label{ex:quiver1}
A quiver with dimension vector defines a representation of a reductive group (see Example \ref{ex:quiver}). Here we define $(V, G)$ to be the representation  defined by the quiver with dimension vector
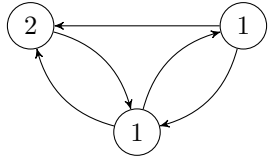
\[
\begin{tikzpicture}[node distance=2cm,->,>=stealth']
\tikzstyle{gauge}=[circle,minimum size=6mm, draw=black]
\tikzstyle{nongauge}=[rectangle,minimum size=6mm, draw=black]

\node[gauge] (1) {$2$};
\node[gauge] (2) [below right of=1] {$1$};
\node[gauge] (3) [above right of=2] {$1$};

\path (3) edge node [above] {} (1)
(2) edge[bend left] node [left ] {} (1)
(2) edge[bend left] node [below right] {} (3)
(1) edge[bend left] node [below left] {} (2)
(3) edge[bend left] node [right] {} (2);
\end{tikzpicture}
\]
where the integers on the vertices indicate their dimension. In particular 
 $V \simeq \AA^8$ 
 %is a direct sum over arrows $a$ of spaces of linear maps $\Hom(\CC^{s(a)}, \CC^{t(a)})$ where $s(a)$ is the integer at the source of $a$ and $t(a)$ is the integer at the target. The group 
 and
\[G = \big( GL(2) \times GL(1) \times GL(1)\big)/\Gm.\]
%acts by conjugation on this space of linear maps, with kernel the diagonal subgroup, and $G$ is the quotient of $G'$ by this kernel.
Let $T$ be the quotient of the diagonal subgroup $T'$ of $GL(2) \times GL(1) \times GL(1)$ by $\Gm$, so $T \subseteq G$ is a maximal torus.

The toric loci $\Sigma(V, T)^W$ and $\omega(V, T)^W$ can be computed with Lemma \ref{lem:toric} as follows. We order the vertices clockwise, beginning with the vertex labeled ``2;" this choice induces a ``standard'' basis of projection characters of $T'$. We choose 
\[
q_1 = (1,0,0,-1) \quad \quad \quad q_2 = (0, 1, -1) \quad \quad \quad q_3 = (0, 0, 1, -1)
\]
for a basis of characters of $T$ (note that each $q_i$ is indeed trivial on the image of $\Gm$, hence a character of the quotient $T = T'/\Gm$). The Weyl group acts on $\chi(T)_\QQ$ by permuting $q_1$ and $q_2$, and so Weyl-invariant elements of $\Char(T)_\QQ$ are given by tuples $(s, s, t)$. In this $(s, t)$ basis, $\Sigma(V, T)^W$ is the entire plane and $\omega(V, T)^W$ is the union of dashed and solid rays in the figure below. The diagonal rays are generated by $(1, -1)$ and $(1, -2)$.
\begin{equation*}
\begin{tikzpicture}[scale=.75]
\draw[gray, fill=gray, opacity=.3] (0,-1.8) rectangle (1.8,1.8);
\draw[<->, dashed] (-2,0)--(2,0);
\node at (2.1, .4) {$s$};
\node at ( -.4, 1.6) {$t$};
\draw[<->] (0,-2)--(0,2);
\draw [->] (0, 0)--(2, -2);
\draw [->, dashed] (0, 0)--(1, -2);
\end{tikzpicture}
\end{equation*}

On the other hand, one can compute, either directly from the definition \ref{def:ss} or using \cite{FPW}, that $\Sigma(V, G)$ is the shaded half plane and $\omega(V, G)$ is the union of the solid rays in the same figure. So $\omega(V, G)$ is properly contained in $\omega(V, T)^W$, even after intersecting with $\Sigma(V, T)^W$.
\end{example}

The next result says that the converse to Lemma \ref{lem:abelianization} \emph{does} hold ``eventually.'' We note that the bound on $r$ in the lemma is far from tight, as is shown by the example of the Grassmannian \ref{ex:grassmannian}.

\begin{proposition}\label{prop:oplus}
Let $(V, G)$ be a representation. If $r >\dim G$ then 
\[
\Sigma(V^{\oplus r}, G) = \Sigma(V^{\oplus r},\, T)^W.
%\quad \quad 
%\text{ and } \quad \quad \omega(V^{\oplus r}, G) = \omega(V^{\oplus r},\, T)^W.
\]
%$V^{\oplus r}$ has $\Sigma_G = \Sigma_T^W$ and the $G$-walls are precisely the $W$-invariant $T$-walls \Rachel{(i.e., equality holds as fans?)}.
\end{proposition}
\begin{proof}
By Lemma \ref{lem:abelianization} it is enough to show $\Sigma(V^{\oplus r},\, T)^W \subseteq \Sigma(V^{\oplus r},\, G)$. For this let $\theta$ be an element of the left hand side; we must show that $G \cdot(V^{\oplus r})^{us}_\theta(T)$ does not contain $V^{\oplus r}$ for $r>\dim(G)$.

We  do this by a dimension count: if (arguing by contrapositive) the orbit $G \cdot(V^{\oplus r})^{us}_\theta(T)$ contains $V^{\oplus r}$, then some irreducible component of $G \cdot(V^{\oplus r})^{us}_\theta(T)$ contains it. It follows from \cite[Tag~0397, 0G2Y]{stacks-project} that the irreducible components of $G \cdot(V^{\oplus r})^{us}_\theta(T)$ are the closures of $G$-orbits of irreducible components of $(V^{\oplus r})^{us}_\theta(T)$, and hence by Corollary \ref{prop:toric} are of the form $\overline{G \cdot Y^{\oplus r}}$ where $Y$ is an irreducible component of $V^{us}_\theta(T)$ and the bar indicates closure. We have a chain of inequalities
\[
r + r \dim Y \leq \dim V^{\oplus r} \leq \dim \overline{G \cdot Y^{\oplus r}} = \dim G \cdot Y^{\oplus r} \leq \dim G + r \dim Y,
\]
where the first inequality holds by the assumption that $\theta$ has a nonempty semistable locus, so $Y \subsetneq V$ and $\dim Y +1 \leq \dim V$. It follows that $r\leq \dim G$.

\end{proof}

\begin{corollary}\label{cor:oplus}
If $(V, G)$ is a Weyl-generic representation, then for $r > \dim G$ there exists $\theta \in \Char(G)_\QQ$ for which the stack quotient $[(V^{\oplus r})^{ss}_\theta(G)/G]$ is a nonempty Deligne-Mumford stack.
\end{corollary}
\begin{proof}
This follows from Proposition \ref{prop:oplus} and Lemma \ref{lem:abelianization}.
\end{proof}

\section{Degeneracies of representations}
Let $H$ be a semisimple group with maximal torus $T$ (for example, $H$ could be a product of simply connected almost-simple groups as in Theorem \ref{thm:structure}) and let $W$ be the associated Weyl group. In this section we introduce an invariant of representations of $H$, called the \emph{degeneracy} of the representation, that we will use to determine Weyl-generic representations.

For $V$ a positive-dimensional representation of $H$, let $\Xi(V)$ be the weights of the $T$-action on $V$. Note that $\Xi(V)$ is a nonempty finite set.

\begin{lemma}\label{lem:exists-a-sum}
    There exist $a_\xi \in \QQ_{\geq 0}$ such that $\sum_{\xi \in\; \Xi(V)} a_\xi = 1$ and $\sum_{\xi \in\; \Xi(V)} a_\xi \xi=0.$
    %There exists a linear relation of form $\sum_i a_i\xi_i=0$ with $a_i\in\QQ_{\geq0}$, not all of which are 0.
\end{lemma}
\begin{proof}
     If all elements of $\Xi(V)$ are zero, we can choose $a_\xi=1$ for one $\xi \in \Xi(V)$ and set all other $a_{\xi'}=0$. 
     
     Otherwise, choose a nonzero element $\theta_0\in\chi(T)_{\QQ}$ of the rational cone generated by the $\xi \in \Xi(X)$. Since the Weyl group acts on $\Xi(V)$, we have that $w \,\cdot\, \theta_0$ is also in this cone for each $w \in W$. Hence for each $w \in W$ we can write 
     $w \,\cdot\, \theta_0=\sum_{\xi \in \;\Xi(V) }a_{w, \xi}\xi$  for some rationals $a_{w, \xi}\geq 0$, not all of which are 0. Therefore we can write 
     \[\theta:=\sum_{w\in W}w\,\cdot\, \theta_0=\sum_{\xi \in \;\Xi(V)}a_\xi'\xi \]  for some rationals $a_\xi'$ with $a_\xi'\geq0$, not all of which are 0. On the other hand, $\theta$ is $W$-invariant by construction, so by the semisimplicity of $H$, we get $\theta=0$. Finally after rescaling we can assume $\sum_{\xi} a_\xi = 1$.
\end{proof}

The lemma allows us to make the following definition.

\begin{definition}\label{def:degeni}
Let $V$ be a positive dimensional representation of $H$. The \emph{degeneracy} of $V$, denoted $\degen(V)$, is the smallest integer $n$ such that there exist weights $\xi_1, \ldots, \xi_{n+1}$ of the $T$-action on $X$ and $a_i \in \QQ_{\geq 0}$ such that
\[
\sum_{i=1}^{n+1} a_i=1 \quad \quad \text{and} \quad \quad \sum_{i=1}^{n+1} a_i\xi_i = 0.
\]
The \emph{realization} of a degeneracy $n$ is the list of pairs $\{(a_i, \xi_i)\}_{i=1}^{n+1}$ such that $\sum_{i=1}^{n+1} a_i\xi_i$ equals zero.
\end{definition}
\begin{remark}\label{rmk:bound}
By Corollary \ref{cor:car} we have that $\degen(V) \leq \rk(H)$. In Proposition \ref{prop:pieces} we classify those representations $V$ that have $\degen(V) = \rk(H)$.
\end{remark}
\begin{remark}
We have that $\degen(V)=0$ if and only if $0$ is a weight of $V$.
\end{remark}
The degeneracy of $V$ can also be characterized as follows. 

\begin{lemma}\label{lem:degen}
Let $V$ be a positive dimensional representation of $H$. 
The degeneracy $\degen(V)$ is the minimum dimension of a subspace of $\Char(T)_\QQ$ that is generated as a cone by weights of $V$.
\end{lemma}
\begin{proof}
The Lemma holds when $0$ is a weight of $V$, so assume this is not the case.

Suppose $U$ is a subspace of $\Char(T)_\QQ$ generated as a cone by a subset $I$ of the weights of $V$, and let $m$ be the dimension of $U$. Let $\theta$ be a nonzero element of this subspace (this exists since no weights are zero). Then $-\theta$ is also in $U$, and we may write 
\[\theta = \sum_{\xi \in I} a_\xi \xi \quad \quad -\theta = \sum_{\xi \in I} b_\xi \xi  \] for some rationals $a_\xi , b_\xi \geq 0$, such that not all $a_\xi$ are zero and also not all $b_\xi$ are zero. Then
\begin{equation}\label{eq:this11}
0 = \sum_{\xi \in I} (a_\xi + b_\xi) \xi
\end{equation}
for some rationals $a_\xi + b_\xi \geq 0$ that are not all zero. By Carath\'eodory's theorem, or rather Corollary \ref{cor:car}, we can arrange for $a_\xi + b_\xi$ to be nonzero for at most $m+1$ of the $\xi \in I$. Dividing \eqref{eq:this11} by the sum of the $a_\xi + b_\xi$ we see that $\degen(V)$ is at most $m$.

Conversely, let $\{(a_i, \xi_i)\}_{i=1}^{n+1}$ be a realization of $\degen(V)$. Since $\sum_{i=1}^{n+1} a_i \xi_i = 0$ the dimension of $\Cone(\{\xi_i\}_{i=1}^{n+1})$ is at most $n$. We claim that $\Cone(\{\xi_i\}_{i=1}^{n+1})$ is a subspace. Indeed, we must have all $a_i \neq 0$ (or $n$ would not be minimal), so we can write
\[
-\xi_{j} = \sum_{\substack{ i=1\\ i \neq j}}^{n+1} (a_i/a_{j}) \xi_i.
\]
\end{proof}

\begin{lemma}\label{lem:degeneracy-properties}
Let $\lambda, \mu$ be dominant weights of $H$ and let $V_\lambda$ and $V_\mu$ be the associated irreducible representations of $H$. There are inequalities
\begin{itemize}
\item[(i)] $\degen(V_\lambda \oplus V_\mu) \leq \min(\degen(V_\lambda), \degen(V_\mu))$
\item[(ii)] $\degen(V_\lambda \otimes V_\mu) \leq \degen(V_\lambda) +\degen(V_\mu)$
\item[(iii)] $\degen(V_\mu) \leq \degen(V_\lambda)$ whenever $\lambda \leq \mu$
\item[(iv)]If $V_\mu$ instead denotes a representation of a second semisimple group $H'$, then $V_\lambda \otimes V_\mu$ is a representation of $H \times H'$, and 
\[
\degen(V_\lambda \otimes V_\mu) \leq \degen(V_\lambda) +\degen(V_\mu).
\]
\end{itemize}
\end{lemma}
\begin{proof}
Part (i) follows from noting that the set of weights of $V_\lambda \oplus V_\mu$ is the union of the sets of weights of $V_\lambda$ and $V_\mu$.

Similarly, part (iii) follows from Lemma \ref{lem:weights-contained}: if $\lambda \leq \mu$ then the set of weights of $V_\lambda$ is contained in the set of weights of $V_\mu$.

For part (iv) (resp. part (ii)) we note that weights of $V_\lambda \otimes V_\mu$ are all vectors of the form $(\xi, \zeta)$ (resp. $\xi + \zeta$) where $\xi$ is a weight of $V_\lambda$ and $\zeta$ is a weight of $V_\mu$. Then (iv) follows from Proposition \ref{P:box-new}. Part (ii) now follows from  (iv) by setting $H' = H$.
\end{proof}

\begin{remark}
The  inequalities in Lemma \ref{lem:degeneracy-properties} are far from tight. For example, if $V$ is any representation of a semisimple group $H$, then in fact
\[
\degen(V \otimes V) \leq \degen(V).
\]
This is because the weights of $V \otimes V$ include $2 \xi$ for all $\xi \in \Xi(V)$. So if $\sum_{i=1}^{n+1} a_i \xi_i$ realizes a degeneracy for $V$, we can multilply this equation by 2 to show that $n$ is an upper bound on $\degen(V \otimes V)$.
\end{remark}

We bound the degeneracies of the representations of each of the simply connected almost-simple groups corresponding to minimal dominant weights.
By Lemma \ref{lem:degeneracy-properties}.(iii), this suffices to bound the degeneracies of all the irreducible representations of these groups. Types $B$--$G$ are handled in Lemma \ref{L:degeneracies}; type $A$ is rather different and is handled in Lemma \ref{lem:SLn}.

\begin{lemma}\label{L:degeneracies}
The degeneracies of the minuscule representations of the groups $B_n, C_n, D_n,$ $ E_6,$ $ E_7,$ $E_8$, $F_4$, and $G_2$ have upper bounds as displayed in Table \ref{T:degeneracies}.\end{lemma}

\begin{table}[h]
\caption{For each simply connected almost-simple group not of type $A$, we list the degeneracies of the minuscule representations.  The $\omega_i$ refer to specific fundamental weights as defined in  \cite[Ch. 6]{Bourbaki} (but see the proof of Lemma \ref{L:degeneracies} for details).}
\label{T:degeneracies}
\begin{tabular}{l|c|c}
group  & minuscule weight $\lambda$ & bound on $\degen(V_\lambda) $  \\ \hline \hline
%  $A_1$ &  $(k,0 ) \quad \quad  k$ odd & $1$ 
%  \\ \hline
%  &  $(k,0 ) \quad \quad  k$ even & 0
% \\ \hline
%  $A_n,\; n \geq 2$ &  (1,0 \ldots, 0) & $n$ 
%  \\ \hline
%  &  (1,1 \ldots,1, 0) & $n$ \\ \hline
%  &  all other $\lambda$ & $n-1$ 
% \\ \hline
 % &  (2,0 \ldots, 0,-1) & $n-1$ 
 % \\ \hline
 % & (2,1 \ldots,1, -1) & $n-1$ 
 % \\ \hline
 % &  (2,1 \ldots,0, 0) & $n-1$ 
 % \\ \hline
 % & $(1^k, 0^{n-k})\;\text{for} \; 1<k\leq n$ & $n-1$ \\ \hline 
 $B_n,\;n \geq 2$ & $\omega_n$& 1 \\ \hline
  $C_n,\;n \geq 3$ &  $\omega_1$ & 1 \\ \hline
  $D_{n},\;n \geq 4$ even & $\omega_1$& 1 \\ \hline
 & $\omega_{n-1}$& 1 \\ \hline
 & $\omega_n$& 1 \\ \hline
  $D_{n},\;n\geq 5$ odd& $\omega_1$ & 1 \\ \hline
 & $\omega_{n-1}$& 3 \\ \hline
 & $\omega_n$& 3 \\ \hline
  $E_{6}$ & $\omega_1$& 2 \\ \hline
% & $\omega_3$& 2 \\ \hline
 & $\omega_6$& 2 \\ \hline
   $E_{7}$ %& $\omega_2$ & 1 \\ \hline
% & $\omega_5$ & 1 \\ \hline
 & $\omega_7$ & 1 \\ \hline
   $E_{8}$ & none &  \\ \hline
   $F_{4}$ & none &  \\ \hline
 $G_2$ &none & 

\end{tabular}

\end{table}

\begin{proof}[Proof of Lemma \ref{L:degeneracies}]
The minuscule weights are listed in \cite[Chapter VI, Exercise 4.15 (p.232)]{Bourbaki}. To compute the degeneracies of the corresponding representations, we will repeatedly use the fact that if $-1$ is in the Weyl group and $\lambda$ is any dominant weight, then since the weights of the representation $V_\lambda$ are as a set invariant under the Weyl action, both $\lambda$ and $-\lambda$ are weights of $V_\lambda$. So the degeneracy of $V_\lambda$ is at most 1 in this case. We will also use Remark \ref{rmk:finding-weights}.\\

% We look up the fundamental weights and Weyl action in \Rachel{cite bourbaki or sage}. We use $e_1, \ldots, e_n$ for the standard basis of $\QQ^n$. \Rachel{set up the logic of each computation, referencing lemmas from intro. explain the thing about -1 in the weyl group.}

\noindent
\textit{Type $B_n.$} 
From \cite[\S VI.4.5]{Bourbaki} we see that $\omega_n$ is not in the root lattice, so 0 is not a weight of $V_{\omega_n}$ and $\degen(V_{\omega_n}) \neq 0$. Since the Weyl group contains  -1 we have $\degen(V_{\omega_n})\leq 1$, so $\degen(V_{\omega_n})=1$.\\

% The roots are
% \[
% e_1-e_2, \quad e_2-e_3, \quad \ldots, \quad e_{n-1}-e_n, \quad e_n.
% \]
% The fundamental weights are
% \[
% \omega_i = e_1 + \ldots + e_{i} \quad \quad i \leq n-1
% \]
% \[
% \omega_n = \frac{1}{2}(e_1 + \ldots + e_n).
% \]
% One checks using the relations in \Rachel{cite bourbaki, or sage??} that $\omega_i \leq \omega_{i+1}$ for $i \leq n-2.$ Moreover, zero is a weight of the representation $X_{\omega_1}$: the Weyl orbit of $\omega_1$ is the set of vectors $\pm e_i, i=1, \ldots n$, and this contains the origin in its convex hull. Since $\omega_1$ is in the root lattice, by \Rachel{cite theorem} we have that zero is a weight of $X_{\omega_1}$. Therefore, the only dominance-minimal fundamental weight is $\omega_n$.

% Since the Weyl group contains the involution $-1$, both $\omega_n$ and $-\omega_n$ are weights of $X_{\omega_n}$. Therefore the degeneracy of this representation is at most 1.\\

\noindent
\textit{Type $C_n$.} From \cite[\S VI.4.6]{Bourbaki} we see that $\omega_1$ is not in the root lattice, so 0 is not a weight of $V_{\omega_1}$ and $\degen(V_{\omega_1}) \neq 0$. Since the Weyl group contains  -1 we have $\degen(V_{\omega_1})\leq 1$, so $\degen(V_{\omega_1})=1$.\\

\noindent
\textit{Type $D_n$.} From \cite[\S VI.4.8]{Bourbaki} the minuscule weights can be represented with the following elements of $\QQ^n$:
\[
\omega_1 = e_1, \quad \quad \quad
\omega_{n-1} = \frac{1}{2}(e_1 + \ldots + e_{n-1} - e_n)\quad \quad \quad
\omega_{n} = \frac{1}{2}(e_1 + \ldots + e_{n-1} + e_n),
\]
where $\{e_i\}_{i=1}^n$ is the standard basis of $\QQ^n$.
The Weyl group permutes the coordinates of a vector in $\QQ^n$ and flips the signs on an even number of coordinates. It follows that $\omega_1$ and $-\omega_1$ are both roots of $V_{\omega_1}$ so $\degen(V_{\omega_1})$ is at most 1, hence is equal to 1 since $\omega_1$ is not a root. 

The degeneracies of $V_{\omega_{n-1}}$ and $V_{\omega_n}$ split into two cases, depending on whether $n$ is even or odd: if $n$ is even, then $-\omega_{n-1}$ (resp. $-\omega_n$) is a root of $V_{\omega_{n-1}}$ (resp. $V_{\omega_n}$) and the degeneracy of these representations is at most $1$. If $n$ is odd then 
%$-\omega_{n-1}$ is in the Weyl orbit of $\omega_n$, hence a weight of $V_{\omega_n}$, but 
we have the following four vectors in the Weyl orbit of $\omega_{n-1}:$
% \[
% \begin{array}{rrrrr
% r
% r}
% v_1 = \omega_{n-1} = ( & 1/2, & 1/2, & \ldots, & 1/2; & 1/2, & -1/2)\\
% v_2 = (& 1/2, & 1/2, & \ldots, & 1/2; & -1/2, & 1/2)\\
% v_3  = (& -1/2, &-1/2, &\ldots, &-1/2; &1/2,& 1/2)\\
% v_4  = (& -1/2,& -1/2, &\ldots, &-1/2; &-1/2, &-1/2)
% \end{array}
% \]
\begin{align*}
v_1 = \omega_{n-1} = & \frac{1}{2}(e_1 + \ldots + e_{n-2})  + \frac{1}{2}e_{n-1}  -\frac{1}{2} e_n\\
v_2 = & \frac{1}{2}(e_1 + \ldots + e_{n-2})  - \frac{1}{2}e_{n-1} +\frac{1}{2} e_n\\
v_3  = -&\frac{1}{2}(e_1 + \ldots + e_{n-2})  + \frac{1}{2}e_{n-1}  +\frac{1}{2} e_n\\
v_4  = -&\frac{1}{2}(e_1 + \ldots + e_{n-2})   -\frac{1}{2}e_{n-1}  -\frac{1}{2} e_n
\end{align*}
% \[
% v_1 = \omega_{n-1} = (1/2, 1/2, \ldots, 1/2; 1/2, -1/2)
% \]
% \[
% v_2 = (1/2, 1/2, \ldots, 1/2; -1/2, 1/2)
% \]
% \[
% v_3  = (-1/2, -1/2, \ldots, -1/2; 1/2, 1/2)
% \]
% \[
% v_4  = (-1/2, -1/2, \ldots, -1/2; -1/2, -1/2)
% \]
Note that $n-2$ is odd.  Since $v_1 + v_2 + v_3 + v_4=0$ we see that the degeneracy of $V_{\omega_{n-1}}$ is at most 3, and likewise since $-v_i$ is in the Weyl orbit of $\omega_n$ the degeneracy of $V_{\omega_n}$ is  at most $3$.

For later use, we  record the following lemma.

\begin{lemma}\label{L:later}
The degeneracies of the minuscule representations of $D_n$, for $n$ odd, can be realized by a set of vectors living in a subspace of $\chi(T)_\QQ$ of dimension 4.
\end{lemma}
\begin{proof}
The minuscule representations of $D_n$  are $V_{\omega_1}, V_{\omega_{n-1}}$, and $V_{\omega_n}$. The span of $v_1, \ldots, v_4$ is a four-dimensional  subspace of $\chi(T)_{\QQ}$ that contains a realization of each degeneracy: 
\begin{itemize}
\item $v_1 + \ldots + v_4$ realizes the degeneracy of $V_{\omega_{n-1}}$
\item  $(-v_1) + (-v_2) + (-v_3) + (-v_4)$ realizes the degeneracy of $V_{\omega_n}$
\item $(v_1+v_4) + (-v_1-v_4)$ realizes the degeneracy of $V_{\omega_1}$
\end{itemize}
For the last claim, note that $\pm(v_1+v_4) = (0, \ldots, 0, \mp 1)$ are in the Weyl orbit of $\omega_1$.
\end{proof}

\noindent
\textit{Type $E_6$.} The weights $\omega_1$ and $\omega_6$ correspond to the dual 27-dimensional representations of $E_6$, so it is enough to bound $\degen(V_{\omega_1})$. Computing the weights of this representation from the data given in \cite[\S VI.4.12]{Bourbaki} is lengthy, but they can also be looked up directly with {\small{\sc SageMath}} using the class \verb+WeylCharacterRing+. In the standard representation of this root system (common to both \cite{Bourbaki} and {\small{\sc SageMath}}), weights of $E_6$ are represented as vectors in $\QQ^8$. If $\{e_i\}_{i=1}^8$ is the standard basis, the weights of $V_{\omega_1}$ include the three vectors
\begin{align*}
v_1 = \omega_1 = -&\frac{2}{3}(e_6+e_7-e_8)\\
v_2=&\frac{1}{3}(e_6+e_7-e_8) + e_5\\
v_3=&\frac{1}{3}(e_6+e_7-e_8) - e_5
\end{align*}
Since $v_1+v_2+v_3=0$ one sees that $\degen(V_{\omega_1}) \leq 2$. (In fact one can check that it is exactly 2.)\\

\noindent
\textit{Type $E_7$.} From \cite[\S VI.4.11]{Bourbaki} we see that $\omega_1$ is not in the root lattice, so 0 is not a weight of $V_{\omega_1}$ and $\degen(V_{\omega_1}) \neq 0$. Since the Weyl group contains  -1 we have $\degen(V_{\omega_1})\leq 1$, so $\degen(V_{\omega_1})=1$.\\

\end{proof}

\begin{lemma}\label{lem:SLn}
Let $V$ be an irreducible representation of $SL(n+1)=A_{n}$. Then $\degen(V)=n$ if $V$ is the standard representation, its dual, or (if $n=1$) if $V$ is one of the representations $Sym^{2\ell+1}(\CC^2)$ for some $\ell\geq 0$ . Otherwise, $\degen(V)\leq n-1$.
\end{lemma}

\begin{proof}
We represent the weights of $V$ as vectors in $\ZZ^{n+1}$ modulo the diagonal. Let $\{e_i\}_{i=1}^{n+1}$ be the standard basis of $\ZZ^{n+1}$. The dominant weights are (weakly) decreasing sequences of integer vectors and the minuscule weights are
\[
\omega_i = e_1 + e_2 + \ldots + e_i \quad \quad \text{for } i=1, \ldots, n.
\]
The positive roots are $e_i-e_j$ for $1\leq i < j \leq n+1$. The Weyl group $S_{n+1}$ acts by permuting the coordinates. We will repeatedly use the observation that $\degen(V) = \degen(V^\vee)$, and that for a dominant weight $\lambda$ the dominant weight corresponding to $V_\lambda^\vee$ can be computed by reversing the order of the coordinates of $-\lambda$. We will also use Remark \ref{rmk:finding-weights}.

We begin by computing $\degen(V_{\omega_1})$ (the standard representation). The weights of $V_{\omega_1}$ are $e_i$ for $i=1, \ldots, n+1$. Linear algebra shows that when $\ell \neq 0$, the only way to write $(\ell,\ell,\ldots, \ell)$ for some $\ell>0$ as a linear combination of the $e_i$'s in $\QQ^{\oplus n+1}$ is to use every $e_i$. So $\degen(V_{\omega_1})=n$. Since $V_{\omega_n} = V_{\omega_1}^\vee$ we get $\degen(V_{\omega_n})=n$ as well.

We next show $\degen(V_{\omega_i}) \leq n-1$ for $i\neq 1, n$. By duality it is enough to consider integers $i \leq (n+1)/2$. Consider the set $\Xi$ of vectors that are (a) supported in the last $n+1-i$ coordinates and (b) have 1's in $i$ consecutive coordinates and 0's elsewhere, where we consider positions $n+1$ and $i+1$ to be consecutive. Then $\xi \in \Xi$ is a weight of $V_{\omega_i}$ (being in the Weyl orbit of $\omega_i$) and 
\[
\sum_{\xi \in \Xi} \xi = i(e_{i+1} + e_{i+2} + \ldots + e_{n+1}), \quad \quad \text{hence} \quad \quad i\omega_i + \sum_{\xi \in \Xi} \xi \equiv 0.
\]
Since $\Xi$ has size $n+1-i$ this shows $\degen(V_{\omega_i}) \leq n+1-i \leq n-1$ (since $i \geq 2$).

By Lemma \ref{lem:degeneracy-properties}.(iii) the preceeding discussion shows $\degen(V_\lambda) \leq n-1$ for all $\lambda$ such that $\omega_i \leq \lambda$ for some $i=2, \ldots, n-1$. We still need to show $\degen(V_\lambda) \leq n-1$ when $\omega_i \leq \lambda$ but $\omega_i \neq \lambda$ for $i=1, n-1$. By duality, it is enough to consider those dominant weights $\lambda$ such that $\lambda = \omega_1 + \alpha$ where $\alpha$ is a nonnegative integral sum of simple roots. In fact it is enough to consider the case where $\alpha$ is a positive root (see e.g. \cite[Thm 2.6]{Stembridge} for a much stronger result). The only dominant weights of this form are
\[
\mu_1 := e_1 + e_2 - e_{n+1} \quad \quad\text{ and } \quad \quad \mu_2 := 2e_1 - e_{n+1}.
\]
However $\mu_2$ is equal to $\mu_1$ plus the simple root $e_1-e_2$, so $\mu_1 \leq \mu_2$ and it is enough to compute $\degen(V_{\mu_1})$. By construction $e_1 \leq \mu_1$ and hence $e_1$, and every element of its Weyl orbit, is a weight of $V_{\mu_1}$. But
\[
\mu_1 + e_3 + \ldots + e_n + 2e_{n+1} \equiv 0,
\]
so $\degen(V_{\mu_1}) \leq n-1.$
\end{proof}

\section{Weyl-genericity: necessary and sufficient conditions}

Let $G$ be a connected complex reductive group with maximal torus $T$ and Weyl group $W$ and let $V$ be a representation. We recall from the introduction the following definition.

\begin{definition}
The representation $(V, G)$ is \emph{Weyl-generic} if $\Sigma(V, T)^W$ is not contained in $\omega(V, T)$.
\end{definition}

\begin{remark}
The property that $(V, G)$ is Weyl-generic does not depend on the choice of maximal torus $T$. If $T'$ is another maximal torus of $G$, then $T$ and $T'$ are conjugate by an element $g \in G$, and conjugation defines an isomorphism of  Weyl groups $W(T, G) \to W(T', G)$ and a Weyl-equivariant isomorphism of vector spaces $\Char(T)_\QQ \to \Char(T')_\QQ$. This isomorphism preserves the weights of $V$ and hence the semistable cone and loci of walls by Lemma \ref{lem:toric}.
\end{remark}

\subsection{Necessary conditions}\label{sec:necessary}
In this section we prove the following theorem, which contains necessary conditions for $(V, G)$ to be Weyl-generic.

\begin{theorem}\label{thm:necessary}
If $(V, G)$ is Weyl-generic, then $G$ has type $A$. Moreover the weights of $V$ as a representation of the center $Z(G)$ span $\Char(Z(G))_\QQ$.
\end{theorem}

The proof of Theorem \ref{thm:necessary} will occupy the remainder of this subsection. By Theorem \ref{thm:structure} there is a central isogeny
\[
H \times D \to G
\]
where $H$ is a product of simply connected almost-simple groups and $D$ is the maximal central torus of $G$. 
By the same theorem the representation $(V, G)$ is Weyl-generic if and only if $(V, H \times D)$ is. 
% \Riku{is there gap here(it is really explained in the above theorem)?} \Rachel{I can't really think of anything else to say here. Do you have a suggestion, or maybe a more specific question?}
Moreover since $D$ is the maximal torus of $Z(G)$, we have that $\Char(Z(G))_\QQ \to \Char(D)_\QQ$ is an isomorphism preserving the weights of $V$. So it suffices to prove the theorem in the case $G = H \times D$ with $H$ and $D$ as above. This is done in Lemmas \ref{lem:typeA} and \ref{lem:fullrank} below. The following lemma will be used in their proof.

\begin{lemma}\label{lem:ample-cone}
Let $G = H \times D$, where $H$ is semisimple with maximal torus $T$ and $D$ is a torus. 
\begin{itemize}
\item[(i)]There is a natural identification of Weyl groups $W(T, H) = W(T \times D, G)$, and
\[
\Char(T \times D)_\QQ^W = \Char(H)_\QQ \times \Char(D)_\QQ.
\]
\item[(ii)] Let $V$ be a representation of $G$ and let $\{\alpha_i \}_{i \in I}$ be the weights of $V$ as a $D$-representation. Then \[\Sigma(T \times D)^W = 0 \times \Cone(\{\alpha_i\}_{i \in I}).\]
\end{itemize}
\end{lemma}
\begin{proof}
For (i), note first that there is an equality
\[
N_{H\times D}(T \times D) = N_H(T) \times D
\]
leading to a canonical identification of Weyl groups $W(T \times D, H \times D) = W(T, H)$. If $(\xi, \alpha)$ is an element of $\Char(T)_\QQ \times \Char(D)_\QQ \simeq \Char(T \times D)_\QQ$, then $w \in W = W(T, H)$ acts by 
\[
w \,\cdot\, (\xi, \alpha) = (w\,\cdot\,\xi, \;\alpha).
\]
Now the result follows from Lemma \ref{lem:spice}.

For (ii) the forward containment follows from (i) and Lemma \ref{lem:toric}. For the backwards containment we must show that $(0, \alpha_i)$ is in $ \Sigma(T \times D)$ for all $i \in I$. Let $V_i \subseteq V$ be the weight space of $\alpha_i$; it is a representation of $H$. By Lemma \ref{lem:exists-a-sum} we can find $T$-weights $\{\xi_j^i\}_{j=1}^m$ of $V_i$ and rationals $a_j \in \QQ_{\geq 0}$ satisfying $\sum_{j=1}^m a_j \xi_j^i = 0$ and $\sum_{j=1}^m a_j = 1$. Then $(\xi_j^i, \alpha_i)$ is a $T \times D$-weight of $V$ for all $j$, and
\[
(0, \alpha_i) = \sum_{j=1}^m a_j(\xi_j^i, \alpha_i),
\]
showing $(0, \alpha_i) \in \Sigma(T \times D).$
\end{proof}

\begin{lemma}\label{lem:typeA}
Let $G = H_1 \times \ldots \times H_M \times D$ and let $V$ be a representation of $G$, where each $H_i$ is simply-connected almost simple, $D$ is a torus, and $H_1$ does not have type $A$. Let $T \subseteq \prod_{i=1}^M H_i$ be a maximal torus. Then $(V, T \times D)$ is not Weyl-generic. \end{lemma}
\begin{proof}
Let $\lambda_1, \ldots, \lambda_m$ be the minuscule weights of $H_1$ (see Table \ref{T:degeneracies}) and let $\lambda_0 = 0$. Let $X_{\lambda_i}$ be the corresponding irreducible representations of $H_1$. The representation $V$ may be written
\[
V = \left(\bigoplus_{i \in I_0} X^{(0)}_i \otimes Y^{(0)}_i\right) \oplus \left(\bigoplus_{i\in I_1} X^{(1)}_i \otimes Y^{(1)}_i\right) \ldots \oplus \left(\bigoplus_{i\in I_m} X^{(m)}_i \otimes Y^{(m)}_i\right)
\]
where each $Y^{(j)}_i$ is a representation of $H_2 \times \ldots \times H_M \times D$ and each $X^{(j)}_i$ is an irreducible representation of $H_1$ whose weights contain the weights of $X_{\lambda_j}$. Suppose $\theta \in \Sigma(T \times D)$ is Weyl-invariant, so by Lemma \ref{lem:ample-cone} the character $\theta$ can be written $(0, \theta')$ where $\theta'$ is a character of $H_2 \times \ldots \times H_M \times D$ and 0 is the trivial character of $H_1$. 
Then $\theta'$ is in the cone generated by the weights of the representations $Y^{(j)}_i$; i.e.,
\[
\theta' = \sum_{\ell=1}^{r_0} c^{(0)}_\ell \xi^{(0)}_\ell + \ldots + \sum_{\ell=1}^{r_m} c^{(m)}_\ell \xi^{(m)}_\ell
\]
where each $\xi^{(j)}_\ell$ is a weight of one of the $Y^{(j)}_i$ and $c^{(j)}_\ell \in \QQ_{\geq 0}$. By Lemma \ref{lem:car} we may assume
\[
r_0 + r_1 + \ldots + r_m \leq r
\]
where $r$ is the rank of $H_2 \times \ldots \times H_M \times D$ (the dimension of the vector space where $\theta'$ lives).

Define
\[
(\theta')^{(j)} = \sum_{\ell=1}^{r_j} c^{(j)}_\ell \xi^{(j)}_\ell  \quad \quad \quad \text{so } \theta' = (\theta')^{(0)} + \ldots + (\theta')^{(m)}.
\]
By Proposition \ref{P:box-new}, the vector $(0, (\theta')^{(j)})$ can be written as a nonnegative linear combination of at most $r_j + d_j$ weights of $V$, where $d_0 = 0$ and $d_j$ is the degeneracy bound for $X_{\lambda_j}$ appearing in Table \ref{T:degeneracies}. Therefore $\theta$ can be written as a nonnegative linear combination of at most
\[
\sum_{j=0}^m r_j + \sum_{j=0}^m d_j = r + \sum_{j=1}^m d_j 
\]
weights.
For all groups except $H_1 = D_5$ and $D_7$, one checks that $\sum_{j=1}^m d_j$ is strictly less than the rank of $H_1$, and the proof is complete.

For $H_1 = D_5$ or $D_7$, we have $\sum_{j=1}^m d_j = 7$. Therefore a priori we have written $\theta$ as a nonnegative linear combination of at most $r + 7$ weights of $V$. However by Lemma \ref{L:later} we can choose these weights so that their projections to $\chi(T_1)_\QQ$ lie in a subspace of dimension 4, where we set $T_1 := T \cap H_1$. Since the projections of these weights to $\chi(T_2 \times \ldots \times T_M \times D)_\QQ$ lie in a subspace of dimension $r$ (namely the whole space), the weights themselves lie in a subspace of $\chi(T \times D)_\QQ$ of dimension $r+4$. By Lemma \ref{lem:car} the vector $\theta$ is a nonnegative linear combination of at most $r+4$ of these weights, and $r + 4 < \rk(G)$ in this case.

\end{proof}

\begin{lemma}\label{lem:fullrank}
%\Riku{Maybe we can rephrase this becasue $G$ is not used anywhere.}
Let $G = H \times D$ where $H$ is semisimple with maximal torus $T$ and $D$ is a torus. Let $V$ be a representation of $G$ and let $\{\alpha_i \}_{i \in I}$ be the weights of $V$ as a $D$-representation. If
\[
\dim \Cone(\{\alpha_i\}_{i \in I}) < \dim \Char(D)_\QQ
\]
then $(V, T \times D)$ is not Weyl-generic.
\end{lemma}
\begin{proof} Let $k$ be the rank of $D$ and let $n$ be the rank of $T$. If the dimension of $\Cone(\{a_i\}_{i \in I})$ is at most $k-1$, then the $T \times D$ weights of $V$ generate a cone of dimension at most $n+k-1$. In particular, by Carath\'eodory's theorem (Lemma \ref{lem:car}) and Lemma \ref{lem:toric} every point of $\Sigma(V, T \times D)$ is contained in $\omega(V, T \times D)$. So $\Sigma(V, T \times D)^W$ is also contained in $\omega(V, T \times D)$.
\end{proof}

\subsection{Sufficient conditions}
Our main tool for constructing Weyl-generic representations is the following definition.
Recall from Remark \ref{rmk:bound} that the degeneracy $\degen(V)$ of a representation $V$ of a semisimple group $H$ is at most the rank of $H$.
\begin{definition}\label{def:degen}
Let $H$ be a semisimple group. A representation $V$ of $H$ is \emph{nondegenerate} if
\[
\degen(V) = \rk(H).
\]
\end{definition}

In Proposition \ref{prop:pieces} we completely classify all irreducible nondegnerate representations (in particular, they only exist when $H$ has type $A$). But first, we show their usefulness by giving one set of sufficient conditions for Weyl-genericity.

\begin{theorem}\label{thm:sufficient}
Let $G=H\times D$ be a product of a semisimple group and a torus. Let $X_+$ be a nondegenerate representation of $H$, let $Y$ be a representation of $D$, and let $Z$ be a representation of $G$. Assume that
\begin{itemize}
\item The cone of weights of $Y$ has full dimension.
\item There is a vector $\nu \in \chi(D)_\QQ^\vee$ such that 
\[
\langle \nu, \alpha^+ \rangle \geq 0 \quad \quad \quad \quad \langle \nu, \alpha^-\rangle < 0
\]
for all $D$-weights $\alpha^+$ of $Y$ and $\alpha^-$ of $Z$.
\end{itemize}
%the cone of weights of $Y$ has the full-dimension, and suppose that the cone of weights of $Y$ and the cone of weights of $Z$ (as a representation of $D$) are disjoint away from the origin.  \Rachel{probably also assume the weights of $Z$ all lie on one side of a fixed hyperplane. In particular, 0 is not a weight of $Z$}
    For any integer $t>0$, let $Z^{(t)}$ denote the representation of $G$ obtained by scaling $D$-weights of $Z$ by $t$. 
    Then, for all $t\gg0$, the representation $(X_+\otimes Y) \oplus Z^{(t)} $ is Weyl-generic.
\end{theorem}

A bound on $t$ such that $(X_+ \otimes Y) \oplus Z^{(t)}$ is Weyl-generic can in theory be computed in examples: see e.g. Example \ref{ex:sufficient}. We present the following corollary as an example of the theorem.
\begin{corollary}\label{cor:sufficient}
Let $G=H\times D$ be a product of a semisimple group and a torus, let $X_+$ be a nondegenerate representation of $H$, and let $Y$ be a representation of $D$ whose cone of weights has full dimension. Then $X_+ \otimes Y$ is Weyl-generic. 
\end{corollary}
\begin{proof}
This follows from Theorem \ref{thm:sufficient} by setting $Z=0$ and $\nu = 0$. Note that in this case the hypothesis that $\langle \nu, \alpha^- \rangle < 0$ for all weights $\alpha^-$ of $Z$ is vacuously true.
\end{proof}

\begin{remark}
Moduli of quiver representations provide many examples of Weyl-generic representations. For instance, quiver flag varieties \cite{craw} are examples of schemes that arise as GIT quotients of certain representations $(V, G)$ by a specific character $\theta$. However, such $V$ almost never contain nondegenerate subrepresentations of $H$ (see Example \ref{ex:quiver}). So moduli of quiver representations are a flavor of Weyl-generic representations somewhat orthogonal to those arising from Theorem \ref{thm:sufficient} and Corollary \ref{cor:sufficient}.
\end{remark}

%\Rachel{rewrote the following remark; please check}
\begin{remark}
In the setting of Theorem \ref{thm:sufficient}, suppose we have fixed $\theta \in \Sigma(X_+\otimes Y, T \times D)^W$. Since there is a containment
\[
\Sigma(X_+\otimes Y, T \times D)^W \subseteq \Sigma((X_+\otimes Y)\oplus Z^{(t)}, T \times D)^W
\]
for all $t$, it makes sense to ask if, by increasing $t$, we may assume $\theta$ is not in a wall of $(X_+\otimes Y)\oplus Z^{(t)}.$ This is not possible in general, as shown in example \ref{ex:walls}. 
\end{remark}

\begin{proof}[Proof of Theorem \ref{thm:sufficient}]
Let $n$ be the rank of $H$, let $k$ be the rank of $D$, and let $T \subseteq H$ be a maximal torus. Furthermore let $A^{+}$ (resp. $A^-$) denote the cone of $D$-weights of $Y$ (resp. $Z$). Note that we are assuming $\dim A^+= k$.

By Lemma \ref{lem:ample-cone} the cone $0 \times A^+$ is contained in 
\[\Sigma((X_+ \otimes Y) \oplus Z^{(t)},\; T \times D)^W\] for any $t$, so to show that $(X_+ \otimes Y) \oplus Z^{(t)}$ is Weyl-generic it is enough to show that $0 \times A^+$ is not contained in any torus wall (for large enough $t$).

Let $\cW$ be a torus wall, i.e. a cone generated by $n+k-1$ weights of $T \times D$ (possibly non-distinct). There are finitely many choices for $\cW$, and we will show that for each possibility we can choose $t$ large enough that $\cW^W:=\cW\cap(0\times\Char(D)_\QQ)$ does not contain $0 \times A^+$. There are two cases, based on the projection of $\cW$ to $\chi(T)_\QQ$ (note that this projection is independent of $t$). If this projected cone spans a vector space of dimension $n$, then since $\cW$ spans a vector space of dimension at most $n+k-1$, the kernel of the projection $\mathrm{span}(\cW) \to \chi(T)_\QQ$ has dimension at most $k-1$. This kernel is precisely $\cW^W$, so in particular the intersection of $\cW$ with $0 \times A^+$ has dimension at most $k-1$ and cannot contain $0 \times A^+$.

For the other case, when the projection of $\cW$ to $\chi(T)_\QQ$ has dimension at most $n-1$, the Weyl-invariant wall $\cW^W$ can have dimension $k$, but we will show it can be ``pushed off'' of $0 \times A^+$ by taking $t$ large enough. Denote the weights that generate $\cW$ by 
\[
(\xi_i^+, \alpha_i^+)_{i \in I} \quad \quad \quad \quad (\xi_j^-, t\alpha_j^-)_{j \in J},
\]
so $\xi^+_i$ is a $T$-weight of $X_+$, $\alpha_i^+$ is a $D$-weight of $Y$, and $\xi^-_j$ and $\alpha^-_j$ are respectively $H$ and $D$-weights of $Z^{(t)}$. Moreover $|I|+|J|=n+k-1$. 

We now compute $\cW^W$. Let $C$ denote the cone given by intersecting the kernel of the matrix
\[
[\xi^+_{i}: \xi^-_j]_{i \in I, j \in J}
\]
with the nonnegative orthant. (The matrix $[\xi^+_i: \xi^-j]_{i \in I, j \in J}$ has columns equal to the $\xi_i^+$ and $\xi_j^-$.) Note that this cone is independent of $t$. Write an element $\bc \in C$ in coordinates as $(c_i^+, c_j^-)_{i \in I, j \in J}$. Then $\cW^W$ is given by transforming this cone under the matrix of $D$-weights:
\[
\cW^W =0\times( [\alpha^+_i: t\alpha^-_j]_{i \in I, j \in J} C).
\]
A key observation is that if $\bc \in C$ then there is some index $j_0 \in J$ for which $c^-_{j_0} \neq 0$. Indeed, if this were not the case, then $\bc = (c^+_i, 0)_{i \in I}$ would define a degeneracy of $X^+$ of dimension at most $n-1$ (using our assumption that the projection of $\cW$ to $\chi(T)_\QQ$ has dimension at most $n-1$ and Carath\'eodory's theorem (Lemma \ref{lem:car})). This is a contradiction since $X^+$ is nondegenerate.

{Let $\{\bc(\ell)\}_\ell$ be a finite set of ray generators of $C$ and for each $\bc(\ell)=(c(\ell)_i^+,c(\ell)^-_j)_{i\in I,j\in J}$ let $j_{\ell}$ be the index of a nonzero coordinate of $\bc(\ell)$.} Then $\cW^W$ is generated by the rays
\begin{equation}\label{eq:newray}
t^{-1}\sum_{i \in I} \frac{c(\ell)^+_i}{ c(\ell)^-_{j_\ell}} (0,\alpha^+_i) + \left( (0,\alpha^-_{j_\ell}) + \sum_{j \in J,\; j \neq j_\ell} \frac{c(\ell)^-_j}{c(\ell)^-_{j_\ell}} (0,\alpha^-_j) \right).
\end{equation}
Observe that the value of $\nu$ on every vector in $0 \times A^+$ is nonnegative. On the other hand, The value of $\nu$ on \eqref{eq:newray} is
\[
t^{-1}\sum_{i \in I} \frac{c(\ell)^+_i}{ c(\ell)^-_{j_\ell}} \langle \nu, \alpha^+_i\rangle + \left( \langle  \nu, \alpha^-_{j_\ell} \rangle + \sum_{j \in J,\; j \neq j_\ell} \frac{c(\ell)^-_j}{c(\ell)^-_{j_\ell}} \langle \nu, \alpha^-_j\rangle \right).
\]
Since all $c(\ell)^{\pm}_{i/j}$ are nonnegative and all $\langle \nu, \alpha^-_{j} \rangle$ are negative, we see that by choosing $t$ large we can make this quantity negative. In particular, for $t$ large enough the value of $\nu$ is strictly negative on all ray generators \eqref{eq:newray} of $\cW^W$. It follows that the cone $\cW^W$ is disjoint from $0 \times A^+$ except possibly at the origin, and in particular does not contain $0 \times A^+$.
\end{proof}

The classification of the nondegenerate representations of semisimple groups uses the following equivalence relation on representations. Let $H = \prod_{i=1}^M SL(n_i)$ and let $V_1, V_2$ be two irreducible representations of $H$. Then we may write
\[
V_i = V_i^{(1)} \otimes V_i^{(2)} \otimes \ldots \otimes V_i^{(M)}
\]
where $V_i^{(j)}$ is an irreducible representation of $SL(n_j)$.
We say $V_1$ and $V_2$ are \emph{equivalent mod $SL(2)$}, writing
\[
V_1 \equiv V_2 \mod SL(2),
\]
if $V_1^{(j)} \simeq V_2^{(j)}$ whenever $n_j \neq 2$.

\begin{proposition}\label{prop:pieces}
Let $H = \prod_{i=1}^M H_i$ be a product of simply connected almost simple groups. If $V$ is a representation of $H$, then $V$ is nondegenerate if and only if
\begin{itemize}
\item[(i)] For all $i$, we have $H_i = SL(n_i)$ for some $n_i \geq 2$, and moreover the $n_j$ are pairwise coprime.
\item[(ii)] If we write $V$ as a direct sum of irreducible representations $\oplus_{j=1}^N V^{(j)}$ of $H$ then $V^{(j)} \cong V^{(j+1)} \mod SL(2)$ for all $j=1, \ldots, N-1$.
\item[(iii)] Each $V^{(j)}$ is a tensor product of nondegenerate representations of the $SL(n_i)$ (see Lemma \ref{lem:SLn}). %\otimes_{i=1}^M V_i^{(j)}$ with each $V_i^{(j)}$ an irreducible representation of $SL(n_i)$, then each $V_i^{(j)}$ is equal to a nondegenerate representation of $SL$
\end{itemize}

% = \prod_{i=1}^M SL(n_i)$ and let $V$ be an irreducible representation of $H$, and write
% \[
% V =  X^{}_1 \otimes \ldots \otimes X^{}_M
% \]
% % \[
% % V = \oplus_{j=1}^N X^{(j)}, \quad \quad \quad \quad X^{(j)} = X^{(j)}_1 \otimes \ldots \otimes X^{(j)}_M)
% % \]
% where $X^{}_i$ is an irreducible representation of $SL(n_i)$. Then $V$ is nondegenerate if and only if the following conditions are satisfied:
% \begin{itemize}
% \item[(i)] The integers $n_i$ are pairwise coprime.
% %\item[(ii)] If $X$ is an irreducible representation of $H_i$ that appears in the list $X^{(1)}_i, X^{(2)}_i, \ldots, X^{(m)}_i$, then the dual representation $X^\vee$ does not appear.
%  \item[(ii)] Each $X^{}_i$ is equal to the standard representation of $SL(n_i)$ or its dual, or (if $n_i=2$) to one of the representations $Sym^{2k+1}(\CC^2)$ or its dual, for some nonnegative integer $k$.
%\end{itemize}
\end{proposition}
\begin{proof}
\noindent
\textit{Forward direction.} 
Assume $V$ is nondegenerate. 
Then for each $j=1, \ldots, N$, we have
\[
\rk(H) = \degen(V) \leq \degen(V^{(j)}) \leq \sum_{i=1}^M \degen(V_i^{(j)}) \leq \sum_{i=1}^M \rk(H_i)
\]
using Definition \ref{def:degen}, Lemma \ref{lem:degeneracy-properties}.(i), Lemma \ref{lem:degeneracy-properties}.(iv), and Remark \ref{rmk:bound}, respectively.
Since $\rk(H) = \sum_{i=1}^M \rk(H_i)$, we must have equality throughout. In fact, by Remark \ref{rmk:bound} we have
\[
\degen(V_i^{(j)}) = \rk(H_i) \quad \quad \quad \text{and} \quad \quad \quad \degen(V^{(j)}) = \rk(H)
\]
for all $j$ and $i$; i.e., $V^{(j)}$ and $V^{(j)}_i$ are nondegenerate representations of their respective groups. From Table \ref{T:degeneracies} and Lemma \ref{lem:SLn} it then follows that $H_i = SL(n_i)$ for some $n_i \geq 2$ and that part (iii) of the Proposition holds.

We now show that (ii) holds. Suppose for contradiction that there exist $j, j'$ with $V^{(j)} \not \equiv V^{(j')} \mod SL(2)$. Without loss of generality we can take $j=1$ and $j'=2$. Then we can write 
\[H = H_{SL2} \times H_= \times H_\neq, \quad \quad V^{(1)} = V_{SL2}^{(1)} \otimes V_= \otimes V_\neq,\quad \quad V^{(2)} = V_{SL2}^{(2)} \otimes V_= \otimes V_\neq^{\vee},
\]
where $H_{SL2}$ is the product of all the $SL(2)$ factors of $H$ and $H_=$ and $H_\neq$ are products of the $SL(n_i)$ factors of $H$ for $n_i \geq 3$; $V^{(j)}_{SL2}$ is a representation of $H_{SL2}$ for $j=1, 2$ and a product of nondegenerate representations of $SL(2)$; and $V_=$ (resp. $V_{\neq}$) is a representation of $H_=$ (resp. $H_{\neq}$) and a product of nondegenerate representations of $SL(n_i)$. Note that $\rk(H_\neq)$ is at least 2. It follows from Lemma \ref{lem:degeneracy-properties} that
\begin{equation}\label{eq:long1} \degen(V) \leq \degen(V^{(1)} \oplus V^{(2)}) = \degen(V_=) + \degen((V^{(1)}_{SL2} \otimes V_{\neq}) \oplus ( V^{(2)}_{SL2} \otimes V^\vee_{\neq})).
\end{equation}
We claim that the representation
\[
(V^{(1)}_{SL2} \otimes V_{\neq}) \oplus ( V^{(2)}_{SL2} \otimes V^\vee_{\neq})
\]
has a degeneracy of at most 1: this is because the nondegenerate representations of $SL(2)$ all contain the standard representation which is self dual, so we can find a weight $\xi$ of $V^{(1)}_{SL2} \otimes V_{\neq}$ such that $-\xi$ is a weight of $V^{(2)}_{SL2} \otimes V_{\neq}^\vee.$ Then from \eqref{eq:long1}, the fact that $V$ is nondegenerate, and Remark \ref{rmk:bound} it follows that
\[
\rk(H) = \degen(V) \leq \degen(V_=) + 1 \leq \rk(H_=) + 1
\]
contradicting the fact that $\rk(H_\neq) \geq 2$.

Finally we show that (i) holds; i.e., that the $n_i$ are pairwise coprime. For this we can replace $V$ with the nondegenerate representation $V^{(1)}$.
Suppose for contradiction that there exist $i, i'$ with $d := \gcd(n_i,n_{i'})>1$. Without loss of generality we can take $i=1$ and $i'=2$. Because (iii) holds, there exist weights $v_1,...,v_{n_1}$ of $V^{(1)}_1$ and $w_1,...,w_{n_2}$ of $V^{(1)}_2$ such that 
\[(n_2/d)(v_1+...+v_{n_1})=0 \quad \quad \text{and} \quad \quad (n_1/d)(w_1+...+w_{n_2})=0.\]
We rescale and break each sum into $d$ pieces:
\[
\underbrace{\frac{n_2}{d}\bigl(v_1 + \dots + v_{n_1/d}\bigr)}_{\nu_1}
+
\underbrace{\frac{n_2}{d}\bigl(v_{n_1/d+1} + \dots + v_{2n_1/d}\bigr)}_{\nu_2}
+\dots+
\underbrace{\frac{n_2}{d}\bigl(v_{n_1 - n_1/d + 1} + \dots + v_{n_1}\bigr)}_{\nu_d} = 0
\]
\[
\underbrace{\frac{n_1}{d}\bigl(w_1 + \dots + w_{n_2/d}\bigr)}_{\omega_1}
+
\underbrace{\frac{n_1}{d}\bigl(w_{n_2/d+1} + \dots + w_{2n_2/d}\bigr)}_{\omega_2}
+\dots+
\underbrace{\frac{n_1}{d}\bigl(w_{n_2 - n_2/d + 1} + \dots + w_{n_2}\bigr)}_{\omega_d} = 0
\]
Note that each $\nu_i$ has $n_1/d$ summands and each $\omega_j$ has $n_2/d$ summands. Consequently, for any $i$ or $j$, the sum of the coefficients of the vectors comprising $\nu_i$ or $\omega_j$ is $n_1n_2/d^2$. By Proposition \ref{P:box-new} it follows that for each $\ell=1, \ldots, d$ we can obtain $(\nu_\ell, \omega_\ell)$ as a nonnegative rational linear combination of $n_1/d + n_2/d - 1$ of the vectors $(v_i, w_j)$, and that at least one of these coefficients is nonzero. Adding together the expressions for all $(\nu_\ell, \omega_\ell)$ we obtain an expression
\[
\sum_{i, j} c_{ij } (v_i, w_j) = 0
\]
with all $c_{ij} \in \QQ_{\geq 0}$, at least one $c_{ij}$ positive, and at most $n_1+n_2-d$ of the $c_{ij}$ nonzero. It follows that
\[
\degen{(V_1^{(1)} \otimes V_2^{(1)})} \leq n_1 + n_2 - d - 1.
\]
Together with Lemma \ref{lem:degeneracy-properties} and Remark \ref{rmk:bound} this implies
\begin{align*}
\degen(V^{(1)}) &= \degen(V_1^{(1)} \otimes V_2^{(1)}) + \sum_{i=3}^M \degen(V_i^{(1)})\\ 
&\leq \rk(H_1) + \rk( H_2) + \sum_{i=3}^M \rk(H_i) - d + 1 = \rk(H) - d + 1.
\end{align*}
It follows that $V^{(1)})$ cannot be nondegenerate unless $d=1$.\\

\noindent
\textit{Backward direction.} We continue with the proof of Proposition \ref{prop:pieces} by proving the backward direction.
Let $H$ be a group and $V$ be a representation satisfying (i)-(iii). We must show that the degeneracy of $V$ is $1+\sum_{i=1}^M (n_i-1)$. 

We first reduce to the case when $V$ is irreducible. If $H$ does not contain an $SL(2)$ factor, then all irreducible summands of $V$ are isomorphic, and it is enough to show that one of these summands is nondegenerate. If $H$ does contain an $SL(2)$ factor, say $n_1=2$, let $\ell$ be the maximum integer such that $Sym^{2\ell+1}(\CC^2)$ is equal to $V_1^{(i)}$ for some $i=1, \ldots, N$. Then the weights of each irreducible summand of $V$ are contained in the weights of $V^{(i)}$, and it is enough to show that $V^{(i)}$ is nondegenerate.

Thus we may assume $V = V^{(1)}$ is irreducible, writing $V = V_1 \otimes \ldots \otimes V_M$ in this case. We represent the weights of $V$ by tuples of vectors $(\xi_1, \ldots, \xi_M)$ where each $\xi_i$ is a weight of $V_i$. We represent $\xi_i$ as an element of $\mathbb{Z}^{n_i}$ modulo the diagonal. To begin, we will assume that $H$ does not contain an $SL(2)$ factor. After presenting the argument in this case we will explain how it can be modified to accommodate the general case.

If all $n_i \geq 3$, there are $n_i$ choices for the entry $\xi_i$ and hence $n_1n_2 \cdots n_M$ weights of $V$. We index these weights with the positions in an $n_1 \times n_2 \times \ldots \times n_M$ array, writing $K = (k_1, \ldots, k_M)$ for a position and $\xi_K$ for the corresponding weight. We identify the possible values of $k_i$ with the integers in the interval $[1, n_i]$.
A nonnegative integral linear combination $\sum c_K \xi_K$ is an assignment of a coefficient $c_K \in \mathbb{Z}_{\geq 0}$ to each position in the array, and so we define an \textit{assignment of coefficients} to be an integer-valued $n_1 \times \ldots \times n_M$ array. The group $S_{n_i}$ acts on the set of these arrays by permuting the codimension-1 subarrays orthogonal to the $i$-axis, and this induces an action of 
\[\Gamma := S_{n_1} \times S_{n_2} \times \ldots \times S_{n_M}\] on the set of assignments of coefficients. Observe that this group action just amounts to reordering the weights of the representations $V_i$.

Suppose we have an assignment of coefficients $\{c_K\}$. The condition $\sum c_K \xi_K\equiv 0$ (in the quotient lattice) is equivalent to requiring, for each $i=1, \ldots, M$, that the sums of codimension-1 slices of the array orthogonal to the $i$-axis are equal to a constant $a_i$ that is independent of the slice.

\begin{definition} 
An assignment of coefficients with $\sum c_K \xi_K\equiv 0$ and associated constants $\{a_i\}_{i=1}^M$ is \emph{small} if after acting by an element of $\Gamma$, there are integers $m_i \leq n_i$, with some $m_i < n_i$, such that for all $i$, all codimension-1 slices of the subarray $[1, m_1] \times [1, m_2] \times \ldots \times [1, m_M]$ perpendicular to the $i$-axis also add to $a_i$.\footnote{Equivalently, the array is block diagonal with two blocks.}
\end{definition}
That $V$ is nondegenerate is an immediate consequence of the next two Lemmas (\ref{lem:1} and \ref{lem:2}). This will conclude our proof of the backwards direction of Proposition \ref{prop:pieces} in the case that $H$ does not contain an $SL(2)$ factor.

\begin{lemma}\label{lem:1}
A nonzero assignment of coefficients cannot be small.
\end{lemma}
\begin{proof}
This is a consequence of the assumption $\gcd(n_i, n_j)=1$.
Assume we have a nonzero assignment of coefficients with $\sum c_K \xi_K \equiv 0$, so we have for each $i=1, \ldots, M$ a positive $a_i$ such that the sum of any codimension-1 slice orthogonal to the $i$-axis is equal to $a_i$.
Adding the coefficients in the entire array, since there are $n_i$ slices perpendicular to the $i$-axis each of size $a_i$, we see that when we add all the entries in the array we get $a_in_i$. We get equations
\[
a_in_i = a_jn_j\quad \quad \quad \quad \text{for all}\; i, j
\]
and since $(n_i, n_j)=1$ we have $a_i = n_i \ell_{ij}$ and $a_j = n_j \ell_{ij}$ for some $\ell_{ij} \in \ZZ$. If the assignment is small for some choice of integers $m_i \leq n_i$, then likewise we have that $a_im_i$ is independent of $i=1, \ldots, M.$ So given a pair $i, j$ we have $a_im_i = a_jm_j$, and substituting $a_i = n_i\ell_{ij}$ and $a_j = n_j\ell_{ij}$ we find $m_in_i = m_jn_j$. Now $(n_i, n_j)=1$ and $m_i \leq n_i$ imply $m_i=n_i$ for all $i$ showing that the assignment is not small.
\end{proof}

% Adding the coefficients in the entire array, we see that the quantity
% \[
% \frac{a_j}{n_j} \prod_{i=1}^M n_i
% \]
% is independent of $j=1, \ldots, M$. Equating pairs of these quantities yields equations
% \[
% a_in_j = a_jn_i\quad \quad \quad \quad \text{for all}\; i, j
% \]
% and since $(n_i, n_j)=1$ we have $a_i = n_i c$ where $c \in \ZZ$ is independent of $i$. If the assignment is small, then similarly we have that
% \[
% \frac{a_j}{m_j} \prod_{i=1}^M m_i
% \]
% is independent of $j$, and substituting $a_i = n_ic$ and equating pairs of the products we have $m_in_j = m_jn_i$ for all $i, j$. Now $(n_i, n_j)=1$ and $m_i \leq n_i$ imply $m_i=n_i$ for all $i$ showing that the assignment is not small.

\begin{lemma}\label{lem:2}
If an assignment of coefficients is not small, it has at least $1+\sum_{i=1}^M (n_i-1)$ nonzero entries.
\end{lemma}
\begin{proof}
Fix an assignment of coefficients. Let $S$ be the set of dimension vectors $(m_1, \ldots, m_M)$ such that after action by some element of $\Gamma$, the given assignment of coefficients has at least $1+\sum_{i=1}^M (m_i-1)$ nonzero entries in the subarray
\[
[1,m_1] \times [1,m_2] \times \ldots \times [1,m_M].
\]
 We note that $S$ is nonempty as it contains $(1, \ldots, 1)$ (the array has some nonzero coefficient, and we can permute this to the ``first'' corner). Moreover $S$ has a partial order given by $(m_1, \ldots, m_M) \leq (m_1', \ldots, m_M')$ if $m_i \leq m_i'$ for all $i$. Let $(m_1, \ldots, m_M)$ be a maximal element of $S$.  If (supposing for contradiction) $m_i < n_i$ for some $i$, then since the assignment of coefficients is not small, there is an index $j \in \{1, \ldots, M\}$ and a value $y \in [1, m_j]$ such that the sum
\[
\sum_{x_1 \in [1, m_1]} \ldots\sum_{x_M \in [1, m_M]} c_{(x_1, \ldots, y, \ldots, x_M)}
\neq a_j.\]
In other words, some $c_K \neq 0$ where $k_j=y$ and $k_i \in [1, m_i]$. Apply the element of $\Gamma$ that permutes $y$ and $m_j+1$. The resulting assignment of coefficients now contains a subarray of size $m_1 \times \ldots \times m_j+1 \times \ldots \times m_M$ with at least $2+\sum_{i=1}^M (m_i-1)$ nonzero entries, contradicting maximality of $(m_1, \ldots, m_M)$.
Hence the maximal element of $S$ is $(n_1, \ldots, n_M)$, giving the required bound on the degeneracy of $V$.
\end{proof}

\begin{remark}
The intuition behind the proof of Lemma \ref{lem:2}, when $H = SL(n_1) \times SL(n_2),$ is as follows. In this case the coefficient array has 2 dimensions. We can check if an assignment of coefficients has at least $1+\sum_{i=1}^2(n_i-1)$ nonzero entries by attempting to put the entries along an ``L'' in the upper left corner (the length of a maximal ``L'' is exactly $(n_1-1) + (n_2-1)+1$). Suppose we have built such an ``L'' of a certain size: this means we can rearrange entries of our matrix to have the following form:
\[
\left[
\begin{array}{c|ccc}
\begin{array}{cccc}
 * & * & \cdots & *\\
 * & 0 & \cdots & 0\\
 \vdots & \vdots & \ddots & \vdots \\
 * & 0 & \cdots & 0
\end{array}
&&
A&\\ \hline
&\\
B && C &\\
&&&
\end{array}
\right]
\]
The condition that the matrix is not small is the condition that $A$  and $B$ are not both zero. If for example $A$ is not zero, then we can extend the horizontal part of our ``L.'' (Note that it is not possible in general to actually use $\Gamma$ to permute nonzero entries into this ``L'' shape.)
\end{remark}

To finish the proof of the backwards direction of Proposition \ref{prop:pieces}, it remains to consider the case when $H$ contains an $SL(2)$ factor. We can employ an argument almost identical to the above. Without loss of generality assume $n_1=2$ and let $V_1 = Sym^{2\ell+1}(\CC^2)$. Then there are $2\ell+2$ choices for $\xi_1$, so a coefficient array has size $2(\ell+1) \times n_2 \times \ldots \times n_M$. Divide this array into two $(\ell+1) \times n_2 \times \ldots \times n_M$ blocks such that the weights of $V_1$ corresponding to the first block are dual to the weights of $V_1$ corresponding to the second block. Keep $\Gamma = S_2 \times S_{n_2} \times \ldots \times S_M$  and let $S_2$ act by permuting the two blocks. The condition on $i=1$ for $\sum c_K \xi_K \equiv 0$ is now that for each block, the sum of all entries is $a_1$.
From here, we reindex the 1st axis to be compatible with the blocks: the $2\ell+2$ positions are labeled with the fractions $(\ell+1)^{-1}, 2(\ell+1)^{-1}, \ldots, 2$ with the labels $(\ell+1)^{-1} \ldots 1$ belonging to the first block. Now Lemmas \ref{lem:1} and \ref{lem:2} are proved as before with the conventions that ``a codimension-1 slice orthogonal to the 1st axis'' is either of our two blocks, and that $m_1$ must always be integral. 

This completes the proof of Proposition \ref{prop:pieces}.
\end{proof}

\subsection{Equivalent conditions}
In this section we apply the previous results to completely classify Weyl-generic representations in two special situations. The first is when $G = H \times D$ and $V$ is a ``simple tensor,'' i.e. a tensor product of an $H$-representation with a $D$-representation. Note that by the discussion in Section \ref{sec:necessary} this completely classifies irreducible Weyl-generic representations of any connected complex reductive group $G$.

\begin{corollary}\label{cor:tensor}
Let $G = H \times D$ where $H$ is semisimple and $D$ is a torus, and let $V$ be a representation of the form $V_H \otimes V_D$ where $V_H$ is a representation of $H$ and $V_D$ is a representation of $D$. Then $V$ is Weyl-generic if and only if $V_H$ is nondegenerate and the $D$-weights of $V_D$ span $\Char(D)_\QQ$.
\end{corollary}
\begin{proof}
Suppose $V_H$ is nondegenerate and the weights of $V_D$ span $\Char(D)_\QQ$. Then $V$ is Weyl-generic by Theorem \ref{thm:sufficient}. Conversely, if $V$ is Weyl-generic, then the weights of $V_D$ span $D$ by Lemma \ref{lem:fullrank}. Now suppose for contradiction that $V_H$ is not nondegenerate. Let $T$ be the maximal torus and let $n$ be the rank of $H$, and let $k$ be the rank of $D$. Then we can find $T$-weights $\xi_1, \ldots, \xi_n$ of $H$ and numbers $a_i \in \QQ_{\geq 0}$, not all zero, such that 
\begin{equation}\label{eq:cor2}\sum_{i=1}^n a_i \xi_i=0.
\end{equation}
Let $\theta \in \Sigma(V, T \times D)^W$; by Lemma \ref{lem:ample-cone} we can write $\theta = (0, \theta_D)$ where $\theta_D \in \Char(D)_\QQ$ is in the cone of weights of $V_D$. It follows from Carath\'eodory's theorem \ref{lem:car} and Lemma \ref{lem:toric} that for some $b_j \in \QQ_{\geq 0}$ and $\alpha_j$ weights of $V_D$ we can write 
\begin{equation}\label{eq:cor3}
\theta_D = \sum_{j=1}^k b_j \alpha_j.\end{equation}
Finally, combining \eqref{eq:cor2} and \eqref{eq:cor3} and using Proposition \ref{P:box-new} we find that $\theta$ is in a cone generated by at most $n+k-1$ weights of $V$. By Lemma \ref{lem:toric} this means $\theta \in \omega(V, T \times D)$.
\end{proof}
\begin{remark}
Let $V$ be a Weyl-generic representation of $G = H \times D$ with $H$ semisimple and $D$ a torus. It is possible that no irreducible summand of $V$ is Weyl-generic. It is also possible that $V$ is not nondegenerate as an $H$-representation. See Example \ref{ex:quiver}.
\end{remark}

The second situation when we can completely classify Weyl-generic representations is when $G$ is isogeneous to the product of  $SL(2)$ and a torus.

\begin{corollary}
Let $V$ be a representation of $G = SL(2) \times D$. Then $V = \bigoplus_{n=1}^\infty Sym^n(\CC^2) \otimes Y_n$ where $Y_n$ is a representation of $D$ (and only finitely many $Y_n$ are nontrivial). Let $\{\alpha_i^-\}_{i \in I}$ be the set of elements of $\Char(D)_\QQ$ arising as the $D$-weight of $Y_{2m}$ for some $m\geq 0$. Then $V$ is Weyl-generic if and only if
\begin{itemize}
\item[(1)] The $D$-weights of $V$ span $\Char(D)_\QQ$, and
\item[(2)]The cone generated by all the $D$-weights of $V$ is not contained in  $\Cone(\{\alpha^-_i\}_{i \in I})$.
\end{itemize}
\end{corollary}

\begin{proof}
Let $k$ be the rank of $D$ and let $T \subseteq SL(2)$ be a maximal torus. We note that by Lemma \ref{lem:toric} a wall in $\omega(V, G)$ is a cone generated by at most $k$ of the $T \times D$ weights of $V$.

If $V$ is Weyl-generic, then (1) holds by Lemma \ref{lem:fullrank}. To prove (2) let $A$ denote the cone generated by all the $D$-weights of $X$ and let $A^- = \Cone(\{\alpha_i^-\}_{i \in I})$.  By Lemma \ref{lem:ample-cone} we have $\Sigma(T \times D)^W = 0 \times A$. On the other hand, we claim $0 \times A^-$ is contained in the walls $\omega(V, G)$: indeed, for any $\alpha \in A^-$ we can write $\alpha =\sum_{j=1}^k a_{i_j} \alpha_{i_j}$ for some $a_{i_j} \in \QQ_{\geq 0}$ (using Carath\'eodory's theorem \ref{lem:car}). Since $0$ is a weight of $Sym^{2m}(\CC^2)$ for all $m \geq 0$ the sum 
\[(0, \alpha) = \sum_{j=1}^k a_{i_j} (0, \alpha_{i_j})\]
expresses $(0, \alpha)$ as a linear combination of at most $k$ of the $T \times D$-weights of $V$. Since $\Sigma(T \times D)^W$ is not contained in the walls, it follows that $A$ is not contained in $A^-$.

Conversely, suppose (1) and (2) hold. 
Let $\cW$ be a torus wall, i.e. a cone generated by $k$ weights $\xi_1, \ldots, \xi_k$ of $T\times D$ (possibly non-distinct). We claim that either $\cW\cap (0\times A)$ spans a linear subspace of dimension $\leq k-1$ or it is contained in $0 \times A^-$; granting this, since $0 \times A$ has dimension $k$ it follows that $V$ is Weyl-generic. As in the proof of Theorem \ref{thm:sufficient}, if the projection of $\cW$ to $\chi(T)_\QQ$ has dimension $1$, then $\cW\cap (0\times A)$ spans a linear subspace of dimension $\leq k-1$ otherwise.
If not, this projection has dimension 0, or in other words the $T$-part of each $\xi_j$ is zero. This means the $T$-part of each $\xi_j$ is a weight of $Sym^{2m}(\CC^2)$ for some $m\geq 0$, and hence the $D$-part of each $\xi_j$ is equal to one of the $\alpha_i^-$. So $\cW$ is contained in $0 \times A^-$.
\end{proof}

\section{Examples}

%%%%%%%%%%%%%%%%
% Begin Example

\begin{example}\label{ex:gr-flop}
This example generalizes the Grassmannian flop introduced in \cite{DS14}. Let $H = \prod_{i=1}^M SL(n_i)$ and let $X$ be a nondegenerate representation of $H$. Let $\CC_a$ denote the 1-dimensional representation of $\Gm$ of weight $a$. Then the self dual representation
\[
V := (X \otimes \CC_a) \oplus (X^\vee \otimes \CC_{-a})
\]
is Weyl generic.

The proof is as follows. Let $n$ be the rank of $H$. It is enough to show that the projection character $H \times \Gm \to \Gm$ is not in any cone generated by $n$ weights of $V$, or equivalently that this character is not in the $\QQ$-linear span of any set of $n$ weights of $X$. So it is enough to prove that, if $\xi_{K_1}, \ldots, \xi_{K_n}$ are weights of $X$ and $c_{K_1}, \ldots, c_{K_n}$ are integers such that $\sum_{i=1}^n c_{K_i}\xi_{K_i}=0$, then $\sum_{i=1}^nc_{K_i}=0$.

This follows from a more careful examination of the proof of the ``backwards direction" in Proposition \ref{prop:pieces}. Indeed, we are considering a generalization of that context where the coefficients $c_K$ are no longer required to be nonnegative, and hence the constant $a_i$ associated to the $i$th axis could be zero. If $a_i \neq 0$ for any $i$, the proofs of Lemmas \ref{lem:1} and \ref{lem:2} hold as before and produce a contradiction.
If some $a_i=0$, then adding all the entries in the array, i.e. summing all the $c_K$'s, we get $0 \cdot n_i = 0$, proving the claim.
\end{example}
%%%%%%%%%%%%%%%%%
% End Example

\begin{example}\label{ex:quiver}
This example highlights differences between quiver representations and the examples arising from Theorem \ref{thm:sufficient}.

Let $Q$ be a quiver (directed graph) with vertex set $Q_0$ and arrows $Q_1$ and let $s, t: Q_1 \to Q_0$ denote the source and target functions, respectively. A \textit{dimension vector} for $Q$ is a vector $\bd \in \NN^{Q_0}$. A quiver and dimension vector $(Q, \bd)$ defines a representation $(V, G)$, where
\[
G = \left(\prod_{v \in Q_0} GL(\bd_v) \right)/\Gm \quad \quad \quad V = \bigoplus_{a \in Q_1} \Hom(\CC^{\bd_{s(a)}}, \CC^{\bd_{t(a)}}),
\]
the quotient defining $G$ is by the diagonal subgroup, and the action of $G$ on $V$ is given by conjugation. If all $\bd_v$ are at least 2, then the derived subgroup of $G$ is isogeneous to $H=\prod_{v \in Q_0} SL(\bd_v)$. So if $Q$ has at least three vertices and all $\bd_v$ are at least 2 then the underlying $H$-representation is not nondegenerate, nor is any direct summand of it.

Nevertheless many $(V, G)$ arising from quivers with dimension vector \textit{are} Weyl-generic. For example, suppose $Q$ is the quiver with $n$ numbered vertices
\[
\begin{tikzcd}
1 \arrow[r] & 2 \arrow[r] & \ldots \arrow[r] & n-2 \arrow[r] & n-1 \arrow[r, shift left = 1.3ex] \arrow[r, draw=none, "\raisebox{+0ex}{\vdots}" description]\arrow[r, shift right = 1.3ex] & n
\end{tikzcd}
\]
and $e$ arrows from $n-1$ to $n$. Let $\bd \in \NN^n$ be a dimension vector with $d_n=1$ and 
\[d_1 \leq d_2 \leq \ldots \leq d_{n-2} \leq d_{n-1} \leq e.\]
Then the corresponding representation $(V, G)$ is Weyl-generic, and the GIT quotient using the determinant character of $G$ is a partial flag variety.
\end{example}

%In the following examples, we represent weights of $SL(n)$ as vectors in $\ZZ^{n+1}$ modulo the diagonal. The dominant weights are decreasing sequences of integer vectors. 

\begin{remark}\label{rmk:notation}
In the next two examples we represent weights of $SL(4)$ as vectors in $\ZZ^{4}$ modulo the diagonal. The dominant weights are decreasing sequences of integer vectors. We use $e_1, \ldots, e_4$ for the standard basis for $\ZZ^{4}$.
\end{remark}
\begin{example}\label{ex:sufficient}
We give an example of how a lower bound on $t$ in Theorem \ref{thm:sufficient} can be computed explicitly. We use the notation in Remark \ref{rmk:notation}.
Let $G = SL(4) \times \Gm$ and for $t \in \ZZ$ let 
\[V(t) := (V_{e_1} \otimes \CC_1) \oplus (V_{e_1+e_2} \otimes \CC_{-t})\]
where $V_\lambda$ is the irreducible representation of $SL(4)$ corresponding to the dominant weight $\lambda$ and $\CC_a$ is the 1-dimensional representation of $\Gm$ of weight $a$. Note that $V_{e_1}$ is the standard representation of $SL(4)$ and has weights $e_1, e_2, e_3, e_4$, while $V_{e_1+e_2}$ has weights $\{e_i+e_j\}_{1\leq i < j \leq 4} $.  Then $\Sigma(V, T \times \Gm)^W$ is a 1-dimensional rational vector space generated by a vector $(\mathbf{0}; 1) \in \Char(T)_\QQ \times \Char(\Gm)_\QQ$, where $T$ is a maximal torus of $SL(4)$ and $\mathbf{0} \in \Char(T)_\QQ$ is the origin. In the notation of the proof of Theorem \ref{thm:sufficient}, we have that $0 \times A^+$ is the ray generated by $(\mathbf{0}; 1)$, and our goal is to find $t$ such that $(\mathbf{0}; 1)$ is not in $\omega(V, T\times \Gm)^W$.

Up to symmetry, the only wall that could intersect $0 \times A^+$ is the one generated by
\[
(\xi^+_1, \alpha^+_1) = (e_1; 1) \quad \quad (\xi^+_2, \alpha^+_2) = (e_2; 1) \quad \quad (\xi^-_1, t\alpha^-_1) = (e_3+e_4; -t).
\]
The cone $C$ is generated by the ray $(c^+_1, c^+_2, c^-) = (1,1,1)$, and so the expression \eqref{eq:newray} for the ray generating $\sW^W$ becomes
\[
t^{-1}\left((\mathbf{0}; 1) + (\mathbf{0}; 1)\right) + (\mathbf{0}, -1) = (\mathbf{0}, 2t^{-1}-1).
\]
One checks that this ray does not contain $0 \times A^+$ as long as $t$ is at least two.
\end{example}

\begin{example}\label{ex:walls}In Theorem \ref{thm:sufficient}, choosing a sufficiently large $t$ ensures the property of being Weyl-generic, but it does not
control the locations of the walls as the following example shows. We use the notation in Remark \ref{rmk:notation}. 
Let $G = SL(4) \times \Gm^4$ and let $\epsilon_1, \ldots, \epsilon_4$ denote the standard basis of projection characters of $\Char(\Gm^4)_\QQ$. Let $T \subseteq SL(4)$ be a maximal torus. Define
\[X_+ = V_{e_1} \quad \quad \quad \text{and} \quad \quad \quad Y = \CC_{\epsilon_1}\oplus\CC_{\epsilon_2}\oplus\CC_{\epsilon_3}\oplus\CC_{\epsilon_4}.\]

Choose any $\alpha$ in $\Sigma(X_+ \otimes Y, T \times D)^W$. Then $\alpha$ is an element of $0 \times \QQ^{\oplus 4}_{\geq 0} \subseteq \Char(T)_\QQ \times \Char(\Gm^4)_\QQ$ and so can be identified with a sum $\sum_{i=1}^4 \alpha_i \epsilon_i$ for some $\alpha_i \in \QQ_{\geq 0}$. We now show how to choose $Z$ so that $\alpha$ is contained in 
$\omega(X_+\otimes Y\oplus Z^{(t)},T\times D)$ for any integer $t>0$. 

For this purpose, we may scale $\alpha$ and thus assume that $\alpha_i$ are even integers. Let $C$ be an even integer satisfying $C>\alpha_i$ for every $i.$ We define 
\[Z:=V_{e_1+e_2}\otimes\CC_{\frac{1}{2}\sum_{i=1}^4 (\alpha_i-C)\epsilon_i}.\] 
%\Rachel{should it say $\frac{1}{2}\sum_{i=1}^4 (\alpha_i-C)\epsilon_i$ instead?}
Note that there is a vector $\nu$ as in the statement of Theorem \ref{thm:sufficient} since we ensured $C>\alpha_i$. 
 For any positive integer $t>0,$ we have
        \begin{align*}
(0,\alpha)=C(e_1,\,\epsilon_1)&+C(e_2,\,\epsilon_2)+C(e_3,\,\epsilon_3)+C(e_4,\,\epsilon_4)\\
&+\frac{1}{t}(e_1+e_2,\,\frac{t}{2}\sum_{i=1}^4 (\alpha_i-C)\epsilon_i)
        +\frac{1}{t}(e_3+e_4,\,\frac{t}{2}\sum_{i=1}^4 (\alpha_i-C)\epsilon_i),
        \end{align*}
so $(0,\alpha)$ is inside $\omega(X_+\otimes Y\oplus Z^{(t)},T\times D)$ regardless of $t$.
        \end{example}

\printbibliography
\end{document}